\documentclass{amsart}
\usepackage[numbers]{natbib}
\usepackage{mathrsfs}
\usepackage{bbm} % bold 1

\usepackage{microtype} %better typography
\usepackage{braket}
\usepackage{adjustbox}

\usepackage{tikz-cd}
\tikzset{
    labl/.style={anchor=south, rotate=90, inner sep=.5mm}
}
\usetikzlibrary{decorations.markings}
\usetikzlibrary{intersections}
\usepackage{pgf,tikz}
\usetikzlibrary{arrows}

\definecolor{qqqqff}{rgb}{0,0,1}

\makeatletter
\tikzcdset{
  open/.code     = {\tikzcdset{hook, circled};},
  closed/.code   = {\tikzcdset{hook, slashed};},
  open'/.code    = {\tikzcdset{hook', circled};},
  closed'/.code  = {\tikzcdset{hook', slashed};},
  circled/.code  = {\tikzcdset{markwith = {\draw (0,0) circle (.375ex);}};},
  slashed/.code  = {\tikzcdset{markwith = {\draw[-] (-.4ex,-.4ex) -- (.4ex,.4ex);}};},
  markwith/.code ={
    \pgfutil@ifundefined%
    {tikz@library@decorations.markings@loaded}%
    {\pgfutil@packageerror{tikz-cd}{You need to say %
      \string\usetikzlibrary{decorations.markings} to use arrows with markings}{}}{}%
    \pgfkeysalso{/tikz/postaction = {
      /tikz/decorate,
      /tikz/decoration={markings, mark = at position 0.5 with {#1}}}
    }
  },
}
\makeatother
\usepackage[english]{babel}
\usepackage{csquotes}
\usepackage{verbatim}

\usepackage[all]{xy}

\usepackage{multicol}

\usepackage{lmodern}
\usepackage[T1]{fontenc}

\usepackage{amssymb}
\usepackage{amsmath}
\usepackage{amsthm}
\usepackage{mathtools}

\usepackage{todonotes} % for todonotes

\usepackage{enumerate}

\usepackage{csquotes}
\usepackage{verbatim}

\usepackage[all]{xy}

\xyoption{2cell}
\UseAllTwocells

\usepackage{multicol}

\usepackage{lmodern}
\usepackage[T1]{fontenc}

\usepackage{xr-hyper}
\usepackage{hyperref}
\hypersetup{
colorlinks=true,
linkcolor=[RGB]{0,0,204},
filecolor=magenta,      
urlcolor=cyan,
citecolor=[RGB]{0,204,0},
}

\usepackage[nameinlink]{cleveref}
\usepackage[normalem]{ulem}
\newtheorem{theorem}{Theorem}[section]
\newtheorem{proposition}[theorem]{Proposition}
\newtheorem{corollary}[theorem]{Corollary}
\newtheorem{lemma}[theorem]{Lemma}
\newtheorem{remark}[theorem]{Remark}

\newtheorem{notation}[theorem]{Notation}

\theoremstyle{definition}
\newtheorem{definition}[theorem]{Definition}
\newtheorem{example}[theorem]{Example}

\newtheorem{situation}[theorem]{Situation}

\numberwithin{equation}{subsection}

\def\lim{\mathop{\mathrm{lim}}\nolimits}
\def\colim{\mathop{\mathrm{colim}}\nolimits}
\def\Spec{\mathop{\mathrm{Spec}}}

\def\gl{\mathop{\mathrm{Gl}}\nolimits}
\def\pr{\mathop{\mathrm{pr}}\nolimits}

\def\Frac{\mathop{\mathrm{Q}}\nolimits}

\def\Pic{\mathop{\mathscr{P}\mathrm{ic}}\nolimits}
\def\pic{\mathop{\mathrm{Pic}}\nolimits}

\def\max{\mathop{\mathrm{max}}\nolimits}

\def\det{\mathop{\mathrm{det}}\nolimits}

\def\dim{\mathop{\mathrm{dim}}\nolimits}
\def\rk{\mathop{\mathrm{rk}}\nolimits}

\def\colim{\mathop{\mathrm{colim}}\nolimits}
\def\lim{\mathop{\mathrm{lim}}\nolimits}

\newcommand{\rB}{\mathrm{B}} % classifying stacks
\newcommand{\cO}{\mathcal{O}}
\newcommand{\cL}{\mathcal{L}}
\newcommand{\cM}{\mathcal{M}}

\usepackage{graphicx} % Required for inserting images

\def\Bun{\mathop{\mathscr{B}un}\nolimits}

\def\Spec{\mathop{\mathrm{Spec}}\nolimits}
\def\Bun{\mathop{\mathscr{B}un}\nolimits}
\def\gm{\mathop{{\mathbb G}_m}\nolimits}
\def\dt{\mathop{{\rm dt}}\nolimits}
\def\lim{\mathop{\mathrm{lim}}\nolimits}
\def\colim{\mathop{\mathrm{colim}}\nolimits}
\def\Spec{\mathop{\mathrm{Spec}}}

\def\wt{\mathop{\mathrm{wt}}\nolimits}

\def\Sdim{\mathop{\mathrm{Sdim}}\nolimits}
\def\Fib{\mathop{\mathrm{Fib}}\nolimits}

\def\gl{\mathop{\mathrm{Gl}}\nolimits}

\def\max{\mathop{\mathrm{max}}\nolimits}

\def\det{\mathop{\mathrm{det}}\nolimits}

\def\dim{\mathop{\mathrm{dim}}\nolimits}
\def\rk{\mathop{\mathrm{rk}}\nolimits}

\def\colim{\mathop{\mathrm{colim}}\nolimits}
\def\lim{\mathop{\mathrm{lim}}\nolimits}

\def\pr{\mathop{\mathrm{pr}}\nolimits}

\newcommand{\isom}[1][]{
	\ar[#1]
	\ar@<0.7ex>@{}[#1]|-*=0[@]{\sim}}

\def\rig{\mathrm{rig}}
\def\red{\mathrm{red}}
\usepackage{pbox}

\title{Distinguishing Algebraic Spaces from Schemes}
\author{Andres Fernandez Herrero, Dario Wei{\ss}mann, and Xucheng Zhang}
\address{(A. Fernandez Herrero) University of Pennsylvania, Department of Mathematics,
209 South 33rd Street,
Philadelphia, PA 19104, USA,
Email: \href{mailto: andresfh@sas.upenn.edu}{andresfh@sas.upenn.edu}}
\address{(D. Wei{\ss}mann) University of Pisa, Department of Mathematics,
Largo Bruno Pontecorvo, 5 – 56127 Pisa, Italy,
Email: \href{mailto: dario.weissmann@posteo.de}{dario.weissmann@posteo.de}}
\address{(X. Zhang) Tsinghua University, Yau Mathematical Sciences Center, Beijing 100084, China, 
Email: \href{mailto: zhangxucheng@mail.tsinghua.edu.cn}{zhangxucheng@mail.tsinghua.edu.cn}}
\date{}

\begin{document}
\begin{abstract}
We introduce local invariants of algebraic spaces and stacks which measure how far they are from being a scheme. Using these invariants, we develop mostly topological criteria
to determine when the moduli space of a stack is a scheme. 
As an application we study moduli of principal bundles on a smooth projective curve.
\end{abstract}

\maketitle
\tableofcontents

\section{Introduction}

Recent progress in the construction of moduli spaces \cite{exthm}
allows for a construction without resorting to geometric invariant theory
(GIT). Via checking two valuative criteria, called $\Theta$- and $S$-completeness,
one obtains a separated good moduli space (in characteristic 0).
In contrast to the classical setting of GIT, quasi-projectivity of such a moduli space is not automatic, and indeed the moduli space may not even be a scheme.

In this paper, we provide certain local invariants which distinguish algebraic
spaces from schemes (see \Cref{subsection: schematic dimension and fibers}). 
We use these invariants to develop criteria for 
schematicness of algebraic spaces
which have a distinct topological flavor.
As a consequence, we obtain a characterization for when a separated moduli space 
of a stack is a scheme in terms of the topology of the stack 
(\Cref{prop: schematic points for adequate moduli spaces}).
The characterization holds for a fairly flexible notion of moduli spaces, which
we call \emph{weak topological moduli spaces} (see \Cref{remark: more general
contexts}). 
Our main interest here is that good \cite{gms}, adequate \cite{ams}, as well as the upcoming topological moduli spaces of Rydh \cite{rydh_tms} are weak topological moduli spaces.

The criteria are derived from the richer geometry 
of algebraic spaces compared to the one of schemes.
Roughly, we find that line bundles or intersections of Weil divisors 
are not behaving as one would expect.
An extreme phenomenon is the case when all (but finitely many) Weil divisors
intersect in the same point, 
a property we call a (local) uniform base point (\Cref{defn: uniform base point}). 
This property is closely related to admitting the point as a universal scheme 
(\Cref{defn: trivial universal scheme}), i.e., being as unschematic as possible.

\begin{theorem}[\Cref{thm: schematically trivial = local ubp for proper locally factorial}]
\label{thm: intro schematically trivial = local ubp}
    If $X$ is a locally factorial irreducible algebraic space proper 
    over an algebraically closed field $k$,
    then a $k$-point $x$ of $X$
    is a local uniform base point if and only if
    $x$ is schematically trivial, i.e., 
    every Zariski open neighborhood of $x$ has $\Spec(k)$ as
    its universal scheme.
\end{theorem}

By definition being a uniform base point is a closed condition. 
Surprisingly, the same holds for the local version in the setting of 
\Cref{thm: intro schematically trivial = local ubp} 
as we can show that schematic triviality is a closed condition, 
see \Cref{corollary: locus of schematically trivial points is closed}.
We emphasize that being a local uniform base point
is a purely topological property, whereas schematic triviality is a priori a geometric
property.

There are of course many non-schematic algebraic spaces which admit non-trivial
morphisms to schemes. However, the general criterion looks quite similar.
Instead of intersecting all Weil divisors 
of an algebraic stack $\mathscr{X}$ locally of finite type
over an algebraically closed field $k$,
we consider for a $k$-point $x$ of $\mathscr{X}$
the intersection of all Weil divisors passing through $x$ that are Zariski-locally
around $x$ cut out by a section, i.e., are effective Cartier and principal
around $x$ (cf. \Cref{defn: local weil divisors}).
This intersection is a closed substack of $\mathscr{X}$ containing $x$.
Taking the union of all irreducible components containing $x$ with the appropriate
substack structure we obtain the \emph{schematic fiber} $\Fib_x(\mathscr{X})$ of $x$.
The name is derived from an alternative description of $\Fib_x(\mathscr{X})$
as the smallest closed substack of the form $\overline{f^{-1}(f(x))}$
for a morphism $f:\mathscr{U}\to S$ of a Zariski open neighborhood
$\mathscr{U}$ of $x$ to a scheme $S$ of finite type (see \Cref{def: schematic fiber}). 
The equivalence of these two
descriptions is shown in \Cref{prop:schematic fiber as base locus}.

For an integral separated algebraic space locally of finite type,
the schematic fiber at $x$ measures the schematicness at $x$. 
Indeed, in \Cref{prop: characterization of algebraic spaces} we show that $\Fib_x(X)$ is of dimension $0$ if and only if $x$ is schematic in the sense of \Cref{defn: schematic point}.
This allows us to distinguish an algebraic space from a scheme:
local separatedness, local triviality of line bundles, or intersection of effective Cartier
divisors may be responsible for the unschematicness of a point. 
Let us state this formally:

\begin{theorem}[\Cref{thm: reason why algebraic space at a point local}]
\label{thm: distinguishing algebraic spaces via abc}
    Let $X$ be an integral quasi-separated algebraic space locally of finite
    type over an algebraically closed field $k$ and let $x \in X(k)$ be a $k$-point. 
    Then $x$ is not schematic if and only if at least one of the following holds:
    \begin{enumerate}[(A)]
        \item There is an open neighborhood $U \subseteq X$ of $x$ and an
        effective Cartier divisor $D \subset U$ such that the corresponding line
        bundle $\cO_U(D)$ is not Zariski-locally trivializable around $x$.
        \item There is a positive dimensional closed subspace $Z \subseteq X$
        containing $x$ such that for every open neighborhood $U \subseteq X$ of
        $x$ and every effective Cartier divisor $D \subset U$ containing $x$, we
        have $Z \cap U \subseteq D$. 
        %In this case, we may take $Z = \Fib_x(X)$.
        \item $X$ is not Zariski-locally separated at $x$, i.e., there is no
        Zariski open neighborhood of $x$ that is separated.
    \end{enumerate}
\end{theorem}

Translating our schematic fiber criterion to an algebraic stack admitting a 
separated moduli space, we obtain the criterion for checking whether a weak topological
moduli space 
(in the sense of \Cref{remark: more general contexts}) is schematic.

\begin{theorem}[\Cref{prop: schematic points for adequate moduli spaces}]
\label{thm: intro criterion schematic moduli}
    Let $\mathscr{X}$ be an integral quasi-separated algebraic stack locally
    of finite type over an algebraically closed field $k$ admitting
    a separated weak topological moduli space $\mathscr{X}\to X$.
    Then a closed $k$-point $x$ of $\mathscr{X}$
    maps to a schematic point of $X$
    if and only if $x$ is the unique closed point in its schematic fiber $\Fib_x(\mathscr{X})$.
\end{theorem}

As an application of this criterion, we study 
the stack $\Bun_G$ of principal $G$-bundles on a smooth projective curve
of genus at least $2$ for a connected reductive group $G$.
The following generalizes \cite[Thm. A]{weissmann-zhang} in two directions:
it includes the more general notion of a weak topological moduli space, 
as well as open substacks in the semistable locus.

\begin{theorem}[\Cref{thm: semistable and schematic}, \Cref{thm: open substacks of semistable G-bundles}]
\label{thm: intro open substacks of principal bundles}
    Let $C$ be a smooth projective connected curve  of genus $g_C \geq 2$ over an algebraically closed
    field. Let $G$ be a connected reductive group.
    Consider an open substack $\mathscr{U}\subseteq \Bun_G^d$ of a connected
    component of the stack of principal bundles admitting
    a separated and schematic weak topological moduli space $\mathscr{U} \to U$.
    
    Then $\mathscr{U}\subseteq \Bun_G^{d,ss}$ is contained in the semistable
    locus and the induced morphism $U\to M^{d,ss}_G$ to the moduli space
    of semistable principal bundles is an open immersion.
\end{theorem}

As an immediate corollary we obtain that all weak topological moduli spaces of
principal $G$-bundles on $C$ (not necessarily schemes or separated)
are birational and in particular of the same dimension, see
\Cref{cor: moduli of principal bundles are birational}. 

\subsection*{Outline of the paper} 
In \Cref{section: the most unschematic points} we introduce the notions
of (local) uniform base points and schematically trivial points
together with some preliminary observations. 
We also provide several examples of algebraic spaces with such points.

In \Cref{section: local ubp vs schematically trivial} we prove \Cref{thm: intro schematically trivial = local ubp}
for spaces with mild singularities, including locally factorial spaces.

In the main \Cref{section: local geometry via maps to schemes} we introduce two measures of unschematicness
of an algebraic space or stack at a point: the schematic fiber and the schematic dimension.
Both of these provide a canonical stratification of a given algebraic stack. 
To compute the schematic fiber at $x$ we show that it coincides with the
intersection of all Weil divisors at $x$ which are Zariski-locally around $x$
effective Cartier and principal. It is also useful to know that schematic fiber 
commutes with products when equipped with the reduced closed substack structure.
We then show that the schematic fiber measures (un)schematicness
together with the translation to moduli spaces 
(\Cref{thm: intro criterion schematic moduli}).

In the last \Cref{section: applications to the moduli of bundles} we apply the results on schematic fibers to
the stack of principal bundles on a smooth projective curve of genus at least $2$
where we show \Cref{thm: intro open substacks of principal bundles}.

\subsection*{Some notation and conventions}
We work over an algebraically closed ground field $k$. Unless otherwise stated,
all algebraic stacks and algebraic spaces are implicitly assumed to be
quasi-separated and locally of finite type over $k$ in all statements and proofs
of this article.

We recall the definition of a schematic point as one of our main properties of interest.
\begin{definition}[Schematic point] \label{defn: schematic point}
    Given an algebraic stack $\mathscr{X}$ and a geometric point $x \in
    |\mathscr{X}|$, we say that $x$ is \emph{schematic} if there is an open substack
    $\mathscr{U} \subseteq \mathscr{X}$ containing $x$ such that $\mathscr{U}$ is a scheme.
\end{definition}

We work frequently with
prime divisors on an algebraic stack. On a smooth stack a prime divisor is cut
out by a line bundle, i.e., the stack is locally factorial. 
While this property allows for stronger statements, it is somewhat subtle
on a stack and we recall the definition.
\begin{definition}[Local factoriality] \label{defn: local factoriality}
    An integral algebraic stack $\mathscr{X}$ is called \emph{locally factorial}
    if every effective reduced Weil divisor $\mathscr{Z} \subset \mathscr{X}$ is
    a Cartier divisor (i.e., its ideal sheaf $\mathcal{I}_{\mathscr{Z}} \subset
    \mathcal{O}_\mathscr{X}$ is a line bundle). 
\end{definition}

The subtlety of this definition arises from the fact that local factoriality is not
smooth (or \'etale) local. If we have an atlas $U \to \mathscr{X}$ such that $U$
is locally factorial, then it follows $\mathscr{X}$ is locally factorial.
However, the local factoriality of $\mathscr{X}$ does not necessarily imply that
any given atlas $U \to \mathscr{X}$ is locally factorial.
\subsection*{Acknowledgments} 
We thank Jarod Alper, Tom\'as G\'omez, Jochen Heinloth, Svetlana
Makarova, David Rydh, and Angelo Vistoli for useful discussions.
We thank George Cooper for comments on a previous version of the manuscript.
We gratefully
acknowledge support from the Simons Center for Geometry and Physics, Stony Brook
University, at which some of the research for this paper was performed.
Dario Weißmann was funded by PRIN 2022
``Geometry of Algebraic Structures: Moduli, Invariants, Deformations''.

\section{The most unschematic points} \label{section: the most unschematic points}
In this section, we introduce the notions of (local) uniform base points
as the intersection of all (but finitely many) codimension 1 closed substacks.
While this is a purely topological property, it is 
closely related to the geometric notion of schematically trivial points,
see \Cref{definition: schematically trivial}. 
These are the most extreme cases of the difference of the topology and 
geometry of algebraic spaces vs. schemes.

\subsection{Uniform base points} 
\begin{definition}[Uniform base locus and uniform base points] \label{defn: uniform base point}
    Let $\mathscr{X}$ be a quasi-separated irreducible algebraic stack locally of
    finite type over $k$. The \emph{uniform base locus} of $\mathscr{X}$ is the
    closed subset $|\mathscr{X}|_{bs} \subseteq |\mathscr{X}|$ obtained by
    intersecting all non-empty codimension 1 closed substacks of $\mathscr{X}$.
    
    A geometric point $x\in |\mathscr{X}|$ is a \emph{uniform base point} if $x \in |\mathscr{X}|_{bs}$. 
    We say that $x \in |\mathscr{X}|$ is a \emph{local} uniform base point
    if there exists an open substack $\mathscr{U} \subseteq \mathscr{X}$ such that $x \in |\mathscr{U}|_{bs}$. 
\end{definition}

\begin{example} \label{example: uniform base-points} \quad 
    \begin{enumerate}[(1)]
        \item An irreducible scheme of finite type over $k$ does not have any
        uniform base points, unless it has dimension $0$.
        \item Let $\mathscr{X}$ be an irreducible algebraic stack with a uniform
        base point $x \in |\mathscr{X}|$. Suppose that there is an open substack
        $\mathscr{U} \subsetneq \mathscr{X}$ containing $x$. Let $\mathscr{Z}$
        be the closed complement of $\mathscr{U}$ in $\mathscr{X}$ equipped with
        its reduced substack structure. Let $\mathscr{X}' :=
        \mathrm{Bl}_{\mathscr{Z}}(\mathscr{X})$ denote the blowup of
        $\mathscr{X}$ at $\mathscr{Z}$. Then, the point $x \in |\mathscr{X}'|$
        is no longer a uniform base point, but it remains a local uniform base
        point.
    \end{enumerate}
\end{example}

\begin{example}
\label{example-vistoli}
    We thank Angelo Vistoli for bringing our attention to the following example
    of a normal proper algebraic space with uniform base points, which based on
    an example of Hironaka.
    
    Let $E$ be a smooth cubic curve in $\mathbb{P}^2$.
    Consider the blowup $\varphi: X\to \mathbb{P}^2$ 
    of $\mathbb{P}^2$ at 10 distinct points $e_1,\dots,e_{10}\in E(k)$. 
    Then the strict transform $\tilde E$ of $E$ in $X$
    has self-intersection $-1$ and is thus contractible by \cite[Cor. 6.12 (a)]{art2}.
    Denote by $\pi: X \to Y$ the contraction of $\tilde E$, and set $y_0 := \pi(\tilde{E}) \in Y(k)$. 
    By construction $Y$  is proper and normal, and 
    only exists as an
    algebraic space a priori. 
    We claim that $y_0$ is a uniform base point of $Y$ if the points $e_1,\dots,e_{10}$ are
    in general position. In this case $Y$ is not a scheme and $y_0$ is its unique uniform base point.

    Assume there was a prime divisor $\tilde C \subset Y$ not containing $y_0$.
    Then $\tilde C$ is an integral curve and is contained in the scheme 
    $Y\setminus \{y_0\} \cong X\setminus \tilde{E}$. 
    We may also view $\tilde C$ as a closed subscheme of $X$, and consider the scheme theoretic image $C := \varphi(\tilde{C})$ of $\tilde C$ under 
    $\varphi: X\to \mathbb{P}^2$. Note that $C$ is still an integral curve in $\mathbb{P}^2$, say of degree $d \geq 1$. 
    Then we have $C\cdot E=3d$. By construction
    $C$ and $E$ can only intersect at the points $e_1,\dots, e_{10}$.
    If $C$ and $E$ intersect non-transversally at some $e_i$, then $\tilde C$
    and $\tilde E$ intersect in $X$, which is impossible.
    Thus, $C$ and $E$ intersect transversally and this implies that $d\leq 3$.
    Choosing the points $e_1,\dots, e_{10}$ such that no $3$ lie on a line, no $6$ lie on a quadric, and no $9$ lie on a cubic
    (other than $E$) we obtain a contradiction.
\end{example}

\begin{example}[{\cite[Cor. 6]{kollar-non-quasiprojectic-moduli}}]
\label{example-kollar}
There are examples of smooth separated algebraic spaces with positive dimensional uniform base locus. Indeed, fix an integer $d \gg 0$. Let $U_d \subseteq |\mathcal{O}_{\mathbb{P}^2}(d)|=\mathbb{P}^N$ 
be the open subscheme consisting of curves $C$ such that
\begin{enumerate}
\item 
$C$ is integral and the genus of its normalization is $>\frac{1}{2}\binom{d-1}{2}$.
\item
$C$ is not fixed by any of the automorphisms of $\mathbb{P}^2$.
\end{enumerate}
Consider the canonical action of $\mathrm{PGl}_3=\mathrm{Aut}(\mathbb{P}^2)$ on $|\mathcal{O}_{\mathbb{P}^2}(d)|=\mathbb{P}^N$. The open subscheme $U_d \subseteq \mathbb{P}^N$ is $\mathrm{PGl}_3$-invariant, and the quotient $U_d/\mathrm{PGl}_3$ is a smooth separated non-schematic algebraic space by \cite[Cor. 6]{kollar-non-quasiprojectic-moduli}. We claim that the algebraic space $U_d/\mathrm{PGl}_3$ has positive dimensional uniform base locus.

Note that, up to positive powers, $\mathcal{O}_{\mathbb{P}^N}(1)$ is the only line bundle on $\mathbb{P}^N$ with non-constant global sections.
Also note that it admits a unique $\mathrm{PGl}_3$-linearization as
$\mathrm{PGl}_3$ has only the trivial morphism to $\mathbb{G}_m$, so there is
only one GIT-stability on $\mathbb{P}^N$ with respect to the
$\mathrm{PGl}_3$-action. Denote by $\mathbb{P}^{N,ss} \subseteq \mathbb{P}^N$ the
open subscheme of semistable points.

Fix an integer $d/\sqrt{2}>m>2d/3$.
Let $Z_d(m) \subseteq U_d$ be the smallest $\mathrm{PGl}_3$-invariant closed subscheme containing the linear subspace of curves with
multiplicity $\geq m$ at $[0:0:1] \in \mathbb{P}^2$. 
Then $Z_d(m)$ is irreducible, non-empty, every curve in $Z_d(m)$
is unstable by \cite[Claim 5]{kollar-non-quasiprojectic-moduli}, and
\[
\dim(Z_d(m)) \geq \dim |\mathcal{O}_{\mathbb{P}^2}(d)|-\binom{m+1}{2}=\binom{d+2}{2}-1-\binom{m+1}{2}.
\]
Hence
\[
\dim(Z_d(m)/\mathrm{PGl}_3) \geq \dfrac{(d+1)(d+2)}{2}-1-\dfrac{m(m+1)}{2}-8>0.
\]

To conclude, we show that $Z_d(m)/\mathrm{PGl}_3$ is contained in the uniform base locus of $U_d/\mathrm{PGl}_3$. This is done by showing the stronger statement
that $Z_d(m)/\mathrm{PGl}_3$ is contained in the uniform base locus of the entire quotient stack
$[\mathbb{P}^N/\mathrm{PGl}_3]$. 

Let $\mathscr{Z} \subset [\mathbb{P}^N/\mathrm{PGl}_3]$ be a prime divisor. 
As the atlas $\mathbb{P}^N \to [\mathbb{P}^N/\mathrm{PGl}_3]$ is smooth, 
the preimage $Z \subset \mathbb{P}^N$ of $\mathscr{Z}$ is a Weil divisor. 
Since $\mathbb{P}^N$ is smooth,
$Z=\mathbb{V}(s)$ is the vanishing locus of some global section $s$ of a line
bundle $\mathcal{O}_{\mathbb{P}^N}(a)$ on $\mathbb{P}^N$ with $a>0$. 
By definition, we find 
$\mathbb{P}^N \setminus \mathbb{P}^{N,ss} \subseteq Z$. 
As $Z_d(m) \subseteq \mathbb{P}^N \setminus \mathbb{P}^{N,ss} \subseteq Z$, 
we have $Z_d(m)/\mathrm{PGl}_3 \subseteq \mathscr{Z}$. 
This shows that $Z_d(m)/\mathrm{PGl}_3$ is contained in the uniform base locus of
$[\mathbb{P}^N/\mathrm{PGl}_3]$.
\end{example}

The following is a reformulation of the property to be a uniform base 
point. Recall that an open substack $\mathscr{U} \subseteq \mathscr{X}$ of an irreducible algebraic stack is 
\emph{big} if the complement $\mathscr{X}\setminus \mathscr{U}$ has codimension at least $2$ in 
$\mathscr{X}$.

\begin{lemma} \label{lemma: zariski neighborhoods around ubp}
    Let $\mathscr{X}$ be an irreducible algebraic stack 
    locally of finite type over $k$ and let $x \in |\mathscr{X}|$ be a geometric point. 
    Then $x$ is a uniform base point if and only if
    every open substack $\mathscr{U} \subseteq \mathscr{X}$ containing $x$ is big. \qed
\end{lemma}

\begin{corollary} 
\label{corollary: ubp implies line bundles are not zariski locally trivializable}
    Let $\mathscr{X}$ be a normal irreducible quasi-separated 
    algebraic stack locally of finite type over $k$, and let $x\in |\mathscr{X}|$
    be a uniform base point. Then the only line bundle on $\mathscr{X}$ that becomes
    Zariski-locally trivial around $x$ is the trivial bundle.
\end{corollary}
\begin{proof}
    This follows from Hartogs's lemma (see \cite[\href{https://stacks.math.columbia.edu/tag/0EBJ}{Tag 0EBJ}]{sp}, which also applies to normal irreducible algebraic stacks), because every Zariski open neighborhood of $x$
    is big (by \Cref{lemma: zariski neighborhoods around ubp}) and $\mathscr{X}$ is normal.
\end{proof}

As a warm-up, we study the notion of adequate moduli space \cite{ams} from the point of view of uniform base points. This is related to the converse GIT statement in \cite[Thm. 11.4]{gms}, and indeed the strategy of proof is the same.
\begin{proposition}\label{prop: ubp-in-unstable}
Let $\mathscr{X}$ be an integral algebraic stack locally of finite type and with affine diagonal over $k$. Suppose that $\mathscr{Y} \subsetneq \mathscr{X}$ is a strictly smaller open substack that admits a schematic adequate moduli space $\pi: \mathscr{Y} \to Y$. If $x \in |\mathscr{X}|$ is a uniform base point, then $x \notin |\mathscr{Y}|$. 
\end{proposition}
\begin{proof}
    Let $y \in |\mathscr{Y}|$. We show that $y$ is not a uniform base
    point of $\mathscr{X}$. Let $U \subseteq Y$ be an affine open neighborhood of
    $\pi(y)$ and consider the preimage $\pi^{-1}(U) \subseteq \mathscr{Y} \subsetneq
    \mathscr{X}$.
    To conclude, we show that the non-empty complement $\mathscr{X}
    \setminus \pi^{-1}(U)$ (equipped with its reduced substack structure) has
    pure codimension 1.
    
    Note that it suffices to show that this is the case when
    we replace $\mathscr{X}$ with an arbitrary quasi-compact open substack
    $\mathscr{V} \subseteq \mathscr{X}$ containing $\pi^{-1}(U)$. 
    Furthermore, we may check the desired codimension estimate after passing to
    a smooth cover $f: V \to \mathscr{V}$, where $V$ is an affine scheme. 
    Since $V$ is affine and $\mathscr{V}$ has affine diagonal, 
    it follows that $f$ is affine. 
    Since $\pi^{-1}(U) \to U$ is adequately affine, the same holds for the composition
    $\pi \circ f: f^{-1}(\pi^{-1}(U)) \to \pi^{-1}(U) \to U$. 
    By \cite[Thm. 4.3.1]{ams}, this implies that $f^{-1}(\pi^{-1}(U))$ is affine,
    and therefore the complement $V \setminus f^{-1}(\pi^{-1}(U))$ has pure
    codimension 1 by \cite[\href{https://stacks.math.columbia.edu/tag/0BCU}{Tag
    0BCU}]{sp}.
\end{proof}

\subsection{Schematically trivial points}

\begin{definition}[Trivial universal scheme] \label{defn: trivial universal scheme}
    Let $\mathscr{X}$ be an algebraic stack over $k$. We say that $\mathscr{X}$ 
    has \emph{trivial universal scheme}
    if any morphism $\mathscr{X} \to Y$ to a $k$-scheme $Y$ %locally of finite type
    factors (uniquely) through a $k$-point $\Spec(k) \to Y$. 
\end{definition}

\begin{definition}[Schematically trivial points]
\label{definition: schematically trivial}
    Let $\mathscr{X}$ be an irreducible algebraic stack locally of finite type over $k$. We say that a geometric point $x \in |\mathscr{X}|$ is \emph{schematically trivial}
    if every open substack $\mathscr{U} \subseteq \mathscr{X}$ containing $x$ has trivial universal scheme.
\end{definition}

Let us present an alternative characterization of schematically trivial points in terms of the Zariski stalk.
\begin{proposition}[Schematic triviality and local functions] \label{prop: schematically trivial points and local functions}
    Let $\mathscr{X}$ be an integral quasi-separated algebraic stack locally of finite type over $k$. Let $x \in |\mathscr{X}|$ be a geometric point. Then, the following are equivalent:
    \begin{enumerate}[(i)]
        \item The point $x$ is schematically trivial.
        \item For any open substack $\mathscr{U} \subseteq \mathscr{X}$ containing $x$, we have $H^0(\mathscr{U}, \mathcal{O}_{\mathscr{U}}) = k$.
        \item We have
         \[
            \colim_{x\in \mathscr{U} \subseteq 
            \mathscr{X}}H^0(\mathscr{U},\mathcal{O}_{\mathscr{U}})=k, 
        \]
        where the colimit runs over all open substacks $\mathscr{U} \subseteq \mathscr{X}$ containing $x$.
    \end{enumerate}
\end{proposition}
\begin{proof}

     The equivalence between (ii) and (iii) is clear. The implication (i) $\Rightarrow$ (ii) is also clear, as global sections of the structure sheaf $\mathcal{O}_{\mathscr{U}}$ correspond to morphisms $\mathscr{U} \to \mathbb{A}^1$, which have to be constant as $\mathscr{U}$ has trivial universal scheme.
     
    For the implication (ii) $\Rightarrow$ (i), assume that $H^0(\mathscr{U},\mathcal{O}_{\mathscr{U}})=k$
    for all Zariski open neighborhoods $\mathscr{U}$ of $x$.
    For any morphism $\pi: \mathscr{U} \to Y$ to a $k$-scheme $Y$, factoring $\pi$ via its scheme-theoretic image, we can assume without
    loss of generality that $\pi$ is schematically dominant and $Y$ is integral. 
    To conclude, we show that $Y = \Spec(k)$. 
    
    Let $U \subseteq Y$ be an affine open subscheme containing $\pi(x)$. First we claim that $U = \Spec(k)$, for which it suffices to show that every
    morphism $f: U \to \mathbb{A}^1$ is constant (i.e., factors through a
    $k$-point of $\mathbb{A}^1$). By assumption (ii), the composition $f \circ \pi: \pi^{-1}(U) \to U \to
    \mathbb{A}^1$ factors uniquely through a $k$-point of $\mathbb{A}^1$. To conclude, it suffices to show
    that $\pi: \pi^{-1}(U) \to U$ remains schematically dominant (this is not
    automatic a priori, since $\pi$ need not be quasi-compact). 
    
    To see this, we fix a smooth schematically dominant morphism $g: X \to \mathscr{U}$ from an integral finite type scheme $X$. Then the composition $\pi \circ g: X \to \mathscr{U}
    \to Y$ is a schematically dominant morphism between integral schemes.
    In particular, the generic point of $X$ maps to the generic point of $Y$. 
    It follows that the restriction
    $\pi \circ g: (\pi \circ g)^{-1}(U) \to \pi^{-1}(U) \to U$ also sends the generic point of $(\pi \circ g)^{-1}(U)$ to the generic point of $U$, 
    and hence is schematically dominant. 
    This implies that $\pi: \pi^{-1}(U) \to U$ is schematically dominant, thus
    concluding the proof of the claim that $U = \Spec(k)$.

    Now we show that $Y = \Spec(k)$. Since $Y$ is integral, this is equivalent to showing that every non-empty affine open subscheme $V \subseteq Y$ satisfies $V = \Spec(k)$. For any such $V$, the intersection $V \cap U \subseteq U$ is non-empty, and so it must be equal to $U = \Spec(k)$. Thus, we have $k \subseteq H^0(V, \cO_V) \hookrightarrow H^0(U \cap V, \cO_{U \cap V}) = k$, which implies that $H^0(V, \cO_V)=k$. Hence $V = \Spec(k)$, as desired.
\end{proof}

Next, we show that every local uniform base point is schematically trivial. In particular, \Cref{example-vistoli} and \Cref{example-kollar} are instances of algebraic spaces with schematically trivial points.
\begin{lemma} \label{lemma: ubp implies schematically trivial}
    Let $\mathscr{X}$ be an integral quasi-separated algebraic stack locally of finite type
    over $k$.
    If $\mathscr{X}$ has a uniform base point, then it has trivial universal scheme.
    In particular, every local uniform base point of $\mathscr{X}$ is schematically trivial.
\end{lemma}
\begin{proof}
    Let $\pi: \mathscr{X} \to Y$ be a morphism to a $k$-scheme $Y$. As in the proof of \Cref{prop: schematically trivial points and local functions}, 
    we may assume that $\pi$ is schematically dominant and $Y$ is integral. 
    To show $Y = \Spec(k)$, we find an affine open subscheme $U \subseteq Y$ with $U = \Spec(k)$. 
    
    Let $x \in |\mathscr{X}|$ be a uniform base point. Let $U \subseteq Y$ be an affine open subscheme containing $\pi(x)$.
    Suppose for the sake of contradiction that $U \neq \Spec(k)$.
    Then $U\setminus \pi(x)$ is not empty and,
    as in the proof of \Cref{prop: schematically trivial points and local functions}, 
    the morphism  $\pi^{-1}(U\setminus \pi(x)) \to U\setminus \pi(x)$ 
    is schematically dominant. In particular $\pi^{-1}(U\setminus \pi(x))$ 
    is non-empty. Since the stack
    $\pi^{-1}(U\setminus \pi(x))$ is locally of
    finite type, it contains a $k$-point and the same holds for $U\setminus \pi(x)$.
    Thus, there exists a (closed) $k$-point $y \in U(k)$ that is distinct from
    $\pi(x)$ in $|U|$. 
    
    Since $U$ is affine, there is a global section 
    $s \in H^0(U, \cO_U)$ that vanishes at $y$ but does not vanish at $\pi(x)$. 
    The preimage $\mathscr{Z} := \pi^{-1}(\mathbb{V}(s))$ of the vanishing
    locus of $s$ is a closed substack of $\pi^{-1}(U)$ that does not contain $x$. 
    Note that $\mathscr{Z}$ is cut out by the non-zero
    global section $\pi^*s \in H^0(\pi^{-1}(U), \cO_{\pi^{-1}(U)})$ on the
    integral stack $\pi^{-1}(U)$, and hence is a Cartier divisor on $\pi^{-1}(U)$. 
    The closure $\overline{\mathscr{Z}} \subset \mathscr{X}$ is 
    a closed substack of pure codimension 1 that does not contain $x$, 
    which contradicts the assumption that $x$ is a uniform base point of $\mathscr{X}$.
\end{proof}

\subsection{Descending schematic triviality}
 \begin{proposition}[Descent of schematic triviality]
  \label{prop: descent of schematic triviality}
    Let $f: \mathscr{X} \to \mathscr{Y}$ be a sche\-matically dominant morphism of integral 
    quasi-separated
    algebraic stacks locally of finite type over $k$. If $x \in |\mathscr{X}|$ is schematically trivial, then $f(x) \in |\mathscr{Y}|$ is schematically trivial.
\end{proposition}
\begin{proof}
    By \Cref{prop: schematically trivial points and local functions} it suffices
    to show that for all open substacks $\mathscr{U} \subseteq \mathscr{Y}$
    containing $f(x)$, we have $H^0(\mathscr{U}, \cO_{\mathscr{U}}) =k$. To this
    end, choose such an open substack $\mathscr{U} \subseteq \mathscr{Y}$ and a
    morphism $g: \mathscr{U} \to \mathbb{A}^1$. By the assumption that $x$ is
    schematically trivial, the composition $g \circ f: f^{-1}(\mathscr{U}) \to
    \mathscr{U} \to \mathbb{A}^1$ is constant (i.e., it factors
    through a $k$-point of $\mathbb{A}^1$). By the same argument as in the proof
    of \Cref{prop: schematically trivial points and local functions}, 
    the morphism $f: f^{-1}(\mathscr{U}) \to \mathscr{U}$ is schematically dominant. 
    It follows that $g: \mathscr{U} \to \mathbb{A}^1$ is constant, as desired.
\end{proof}

For our next descent property, we recall the notion of a categorical space.
\begin{definition}[Categorical space] \label{defn: categorical space}
    Let $\pi: \mathscr{X} \to X$ be a morphism from an algebraic stack to an algebraic space over $k$. 
    We say that $\pi$ is a \emph{categorical space} if it is initial among maps
    from $\mathscr{X}$ to algebraic spaces over $k$.
\end{definition}

 \begin{corollary}[Schematic triviality and categorical spaces] %[Descent of 
\label{prop: schematic triviality descends for uniform categorical spaces}
    Let $\mathscr{X}$ be an integral quasi-separated
    algebraic stack locally of finite type over $k$. Let $\pi: \mathscr{X} \to X$ be a categorical space such that $X$ is locally of finite type over $k$. If $x \in |\mathscr{X}|$ is schematically trivial, then $\pi(x) \in |X|$ is schematically trivial.
\end{corollary}
\begin{proof}
    By considering the scheme theoretic image of $\pi$, the universal property
    of $\pi$ implies that $\pi$ is schematically  dominant and $X$ is integral. 
    Then we can apply \Cref{prop: descent of schematic triviality} to conclude.
\end{proof}

\begin{remark}
    For any integral quasi-separated
    algebraic stack $\mathscr{X}$ locally of finite type over $k$, we will see that the locus $\mathscr{Z} \subseteq \mathscr{X}$ of schematically trivial $k$-points is closed (\Cref{corollary: locus of schematically trivial points is closed}). In particular, by \Cref{prop: schematic triviality descends for uniform categorical spaces}, if $\mathscr{U} \subseteq \mathscr{X}$ is an open substack which admits a schematic categorical space that is not equal to $\Spec(k)$, then we have $\mathscr{U} \subseteq \mathscr{X} \setminus \mathscr{Z}$. Hence, $\mathscr{X} \setminus \mathscr{Z}$ is the maximal open substack that could possibly admit a non-trivial schematic categorical space.
\end{remark}

\section{Local uniform base points versus schematically trivial points} \label{section: local ubp vs schematically trivial}
In many cases, the notion of a schematically trivial point and a local uniform base point agree. This is a priori surprising since being a local uniform base point is a purely topological property, whereas schematic triviality is a geometric property. Recall that every local uniform base point is schematically trivial (see \Cref{lemma: ubp implies schematically trivial}). For the converse, 
the strategy of proof is the same in all cases and relies foremost 
on the theorem of the base together with the observation that a schematically trivial
point of a normal proper algebraic space implies the triviality of $\pic^0$.

\subsection{The locally factorial case}
In this subsection, we prove the converse for certain types of stacks
(\Cref{thm:main}), including the case of proper locally factorial algebraic spaces
(\Cref{thm: schematically trivial = local ubp for proper locally factorial}).

\begin{theorem}\label{thm:main}
    Let $\mathscr{X}$ be a locally factorial (\Cref{defn: local factoriality})
    irreducible quasi-separated algebraic stack over $k$ with finitely generated Picard group.
    Let $x\in \mathscr{X}(k)$ be a $k$-point. Then
    $x$ is a local uniform base point if and only if $x$ is schematically trivial.
\end{theorem}
\begin{proof}
    We already know that local uniform base points are schematically trivial 
    by \Cref{lemma: ubp implies schematically trivial}.
    To show the converse, suppose that $x$ is schematically trivial.
    Assume by way of contradiction that $x$ was not a local uniform base point. 
    Then there exist countably many pairwise distinct prime divisors 
    $\mathscr{Z}_i$ for $i\in\mathbb{N}$ on $\mathscr{X}$, not containing $x$.
    As $\mathscr{X}$ is locally factorial, 
    each $\mathscr{Z}_i$ is the vanishing locus of some global section $s_i$ of a line bundle $L_i$ on $\mathscr{X}$. 

    By assumption $\pic(\mathscr{X})$ is a finitely generated abelian group. 
    Thus, there is a non-trivial relation among these line bundles $L_i$, say
    \[
        L:=\bigotimes_{i \in I} L_i^{\otimes n_i} \cong \bigotimes_{j \in J} L_j^{\otimes n_j},
    \]
    for some finite disjoint subsets $I,J \subset \mathbb{N}$ and integers $n_i,n_j>0$.
    Note that the line bundle $L$ admits two global sections $\otimes_{i \in I} s_i^{\otimes n_i}$ and $\otimes_{j \in J} s_j^{\otimes n_j}$ cutting out distinct effective Cartier divisors on $\mathscr{X}$, so they are linearly independent in $H^0(\mathscr{X},L)$. This shows that $h^0(\mathscr{X},L)\geq 2$.
    By construction $L$ becomes trivial on the open neighborhood 
    $\mathscr{U}:=\mathscr{X} \setminus \bigcup_{i \in I} \mathscr{Z}_i$ of $x$. 
    We conclude $h^0(\mathscr{U},\mathcal{O}_{\mathscr{U}})=h^0(\mathscr{U},L|_{\mathscr{U}})\geq 2$ contradicting the assumption
    of $x$ being schematically trivial by \Cref{prop: schematically trivial points and local functions}.
\end{proof}

\begin{theorem} \label{thm: schematically trivial = local ubp for proper locally factorial}
    Let $X$ be a locally factorial irreducible algebraic space proper over $k$. 
    Let $x \in X(k)$ be a $k$-point.
    Then $x$ is a local uniform base point if and only if $x$ is schematically trivial.
\end{theorem}
\begin{proof}
    By \Cref{thm:main}, it suffices to show that $\pic(X)$ is finitely
    generated. By the theorem of the base for proper algebraic spaces, see
    \cite[Theorem 3.4.1]{brochard-finiteness}, we know that the N\'eron-Severi group
    $\pic(X)/\pic^0(X)$ is finitely generated. Hence, it suffices to show that
    $\pic^0(X)$ is trivial.

    Since $X$ is normal, irreducible, and proper, the Picard functor
    $\underline{\pic}^0(X)$ is represented by a proper group scheme over $k$,
    see \cite[Th\'eor\`eme 4.3.1]{brochard-proper}.
    %(for properness see the proof in \cite[Chpt. 9,Thm. 5.4]{fga}, which applies without modification to algebraic spaces). 
    Therefore its reduction $\underline{\pic}^0(X)_{\red}$ is an abelian variety. The choice of the base-point $x \in X(k)$ yields a Poincar\'e bundle $L$ on $X \times \underline{\pic}^0(X)_{\red}$, which defines a morphism $X \to \underline{\pic}^0(X)_{\red}^{\vee}$ to the dual abelian variety. Since $x$ is schematically trivial, the morphism $X \to \underline{\pic}^0(X)_{\red}^{\vee}$ factors through a $k$-point of $\underline{\pic}^0(X)_{\red}^{\vee}$, which by construction must correspond to the trivial line bundle on $\underline{\pic}^0(X)_{\red}^{\vee}$. We conclude that the Poincar\'e bundle $L$ on $X \times \underline{\pic}^0(X)_{\red}$ is the pullback of a line bundle from $X$, which must be the trivial line bundle. By the universality of $L$, this means that $\cO_X$ is the unique line bundle in $\pic^0(X)$, as desired.
\end{proof}

\begin{remark}
    Using \Cref{prop: schematically trivial points and local functions}, we can interpret \Cref{thm: schematically trivial = local ubp for proper locally factorial} as follows: for a locally factorial proper
    algebraic space the Zariski stalk around a point carries the
    information for the point to be a local uniform base point.
\end{remark}

\subsection{The case of algebraic spaces with mild singularities}

\begin{proposition} \label{prop: mild singularities}
    Let $X$ be an integral separated algebraic space of finite type over $k$, and suppose that $k$ has characteristic $0$. Let $x \in X(k)$ be a $k$-point. Assume that at least one of the following holds:
    \begin{enumerate}[(i)]
        \item $X$ is smooth at $x$ or has a normal isolated singularity at $x$.
        \item $X$ is proper, normal and there is a birational proper morphism $\pi: Y \to X$ such that $Y$ is smooth and $\mathbb{R}^1\pi_*\cO_Y =0$.
        \item $X$ is proper and has rational singularities (i.e., for any surjective \'etale morphism $U \to X$ from a scheme $U$, we have that $U$ has rational singularities).
    \end{enumerate}
    Then, $x$ is a local uniform base point if and only if $x$ is schematically trivial.
\end{proposition}
\begin{proof}
By \Cref{lemma: ubp implies schematically trivial},
it suffices to show that if $x$ is schematically trivial, then $x$ is a local
uniform base point.

\noindent \textit{Case (i):}  Since the statement is local on $X$, we may use Nagata
compactification \cite[Theorem 1.2.1]{nagata_compactification_clo} 
to assume that $X$ is proper and normal.
In both cases we
may resolve the singularities away from $x$ and assume that $X$ is smooth away
from $x$ since the characteristic of the ground field is $0$. 

The same argument as in \Cref{thm: schematically trivial = local ubp
for proper locally factorial} shows that $\pic(X)$ is finitely generated. Now we
may apply the same argument by contradiction as in the proof of \Cref{thm:main},
using the fact that any effective Weil divisor that does not contain $x$ is cut
out by a global section of a line bundle.

\noindent \textit{Case (ii):}
We first show that if $x$ is schematically trivial then 
$\pic^0(Y)$ is trivial. Indeed, we have already seen that if $X$
has a schematically trivial point, then $\pic^0(X)$ is trivial
via the Albanese morphism, using that $\pic^0(X)_{\red}$ is an abelian variety. 
In characteristic $0$, we have that $\pic^0(X)=\pic^0(X)_{\red}$ and an analogous equality for $Y$.
Consider the morphism of abelian varieties $\pi^{\ast}:\pic^0(X) \to \pic^0(Y)$ induced by pulling back line bundles. The induced morphism on tangent spaces at the identity 
coincides with the natural map $0=H^1(X,\mathcal{O}_X)\to H^1(Y,\mathcal{O}_Y)$.

Via the Leray spectral sequence and the vanishing of $\mathbb{R}^1\pi_{\ast}\mathcal{O}_Y$,
we conclude that $H^1(X,\mathcal{O}_X)\cong H^1(Y,\mathcal{O}_Y)$.
Since $\pic^0(Y)$ is smooth of dimension $h^1(Y,\mathcal{O}_Y)=0$ we conclude
the desired triviality.

Note that the birational morphism $\pi$ is an isomorphism over a big open $U \subseteq X$, 
as we can extend the inverse to all codimension $1$ points of the normal
space $X$ using the valuative criterion of properness.

To show that $x$ is a local uniform base point we assume by way of contradiction
that $x$ is not contained in countably many prime divisors $Z_i$ for $i \in \mathbb{N}$ on $X$.
Consider the underlying reduced space $\tilde Z_i$ of the strict transform of $Z_i$ in $Y$. This coincides with equipping the closure of $Z_i \cap U$ in $Y$ with its reduced subspace structure. Since $Y$ is smooth, each $\tilde Z_i$ is the vanishing locus of a global section $s_i$ of a line bundle $L_i$ on $Y$. 

Applying the theorem of the base to $Y$ we obtain a non-trivial relation 
among these line bundles $L_i$, say  
\[
L:=\bigotimes_{i \in I} L_i^{\otimes n_i} \cong \bigotimes_{j \in J} L_j^{\otimes n_j},
\]
for some finite disjoint subsets $I,J \subset \mathbb{N}$ and integers
$n_i,n_j\geq 1$. As in the proof of \Cref{thm:main} we have that $h^0(Y,L) \geq 2$ and hence $h^0(X,j_*L_{|U}) \geq 2$, where $j:U\to X$ is the open immersion. By construction $j_*L_{|U}$ is trivial on the open neighborhood $V:=X\setminus \bigcup_{i \in I}Z_i$ of $x$. Since $V\cap U \subseteq V$ is big open, we conclude 
\[
h^0(V,\mathcal{O}_V)=h^0(V,(j_*L_{|U})_{|V})=h^0(V\cap U,L_{|U}) \geq h^0(X,j_*L_{|U}) \geq 2,
\]
contradicting the assumption of $x$ being schematically trivial by \Cref{prop: schematically trivial points and local functions}.

\noindent \textit{Case (iii):}
This follows directly from (ii) since $X$ having rational singularities implies
the existence of
a resolution as in (ii).
\end{proof}

\begin{corollary}
    Suppose that $X$ is a separated normal algebraic space of finite type over $k$, and $X$ is equi-dimensional of dimension $2$. Then a $k$-point $x \in X(k)$ is a local uniform base point if and only if it is schematically trivial.
\end{corollary}

\begin{proof}
    This is an immediate consequence of \Cref{prop: mild singularities}, since all points are either smooth or isolated singularities.
\end{proof}

\section{Local geometry of stacks via maps to schemes} 
\label{section: local geometry via maps to schemes}

Let $\mathscr{X}$ be an integral quasi-separated algebraic stack
locally of finite type over $k$, and let $x \in |\mathscr{X}|$.
In this section we introduce some local invariants of the stack $\mathscr{X}$ at $x$
measuring the unschematicness of $\mathscr{X}$ at $x$.
These invariants are the \emph{schematic dimension} and the \emph{schematic fiber}, 
see \Cref{def: schematic dimension} and \Cref{def: schematic fiber} respectively.
Both provide a way to stratify $\mathscr{X}$, 
see \Cref{prop: stratification by schematic dimension}.

The schematic fiber at $x$ can be seen as an ``excess intersection'' of Weil-divisors
that are Cartier at $x$, \Cref{prop:schematic fiber as base locus},
and provides an intrinsic way to distinguish an algebraic space
from a scheme by its geometry, \Cref{prop: characterization of algebraic spaces} and 
\Cref{thm: reason why algebraic space at a point local}. 
As the geometry of an algebraic stack admitting a moduli space
already encodes the geometry of the moduli space, this allows for a criterion
to check the schematicness of a separated moduli space: 
the moduli space is schematic at the image of a closed point $x$ of $\mathscr{X}$ 
if and only if the schematic fiber at $x$ has $x$ as its unique closed point, 
\Cref{prop: schematic points for adequate moduli spaces}.

\subsection{Schematic dimension and schematic fibers} 
\label{subsection: schematic dimension and fibers}

\begin{definition}[Finite type algebra at a point]
\label{defn: finite type algebra at x}
    Let $\mathscr{X}$ be an integral quasi-sepa\-rated algebraic stack locally
    of finite type over $k$.
    A \emph{finite type algebra at} a geometric point $x \in |\mathscr{X}|$ is the data of a pair $(\mathscr{U}, A)$, 
    where $\mathscr{U} \subseteq \mathscr{X}$ is an open substack containing $x$ 
    and $A \subseteq H^0(\mathscr{U}, \cO_{\mathscr{X}})$ is a finitely generated
    $k$-subalgebra. 

    Notice that $A$ is always an integral domain. 
    For any finite type algebra $(\mathscr{U}, A)$ at $x$, 
    we denote by $f_A: \mathscr{U} \to \Spec(A)$ the induced schematically dominant morphism.
\end{definition}

\begin{lemma}[Dimension bound for finite type algebras] \label{lemma: dimension bound}
    Let $\pi: X \to \mathscr{X}$ be a schematically dominant smooth morphism from an
    integral algebraic space $X$ of finite type
    to an integral quasi-separated algebraic stack $\mathscr{X}$ locally of finite
    type over $k$.
    %over $k$. 
    For any finite type algebra $(\mathscr{U},A)$ at $x \in |\mathscr{X}|$,
    we have $\dim(\Spec(A)) \leq \dim(X)$.
\end{lemma}
\begin{proof}
    We have a schematically dominant morphism $\pi^{-1}(\mathscr{U}) \to \mathscr{U} \to \Spec(A)$,
    and therefore $\dim(\Spec(A)) \leq \dim(\pi^{-1}(\mathscr{U})) = \dim(X)$.
\end{proof}

\begin{definition}[Schematic dimension]
\label{def: schematic dimension}
    Given an integral quasi-separated
    algebraic stack $\mathscr{X}$ locally of finite type over $k$ and a geometric point $x \in |\mathscr{X}|$, the 
    \emph{schematic dimension} $\Sdim_x(\mathscr{X})$ of $\mathscr{X}$ at $x$ is defined to be the maximum
    \[ 
        \Sdim_x(\mathscr{X}) := \underset{(\mathscr{U}, A)}{\max} \dim(\Spec(A)),
    \]
    where the maximum ranges over all finite type algebras $(\mathscr{U},A)$ at $x$. This maximum is achieved, by \Cref{lemma:
    dimension bound}.
\end{definition}

\begin{example}
    A geometric point $x\in |\mathscr{X}|$ has schematic dimension $0$ 
    if and only if $x$ is schematically trivial.
\end{example}

Since $\mathscr{X}$ is integral, given any two open neighborhoods 
$\mathscr{U} \subseteq \mathscr{V}$ of $x$, restriction induces an injection of rings
$H^0(\mathscr{V}, \cO_{\mathscr{X}}) \hookrightarrow H^0(\mathscr{U}, \cO_\mathscr{X})$. 
In particular, for any finite type algebra $(\mathscr{V},A)$
at $x$, we have a canonical restriction $(\mathscr{U}, A)$, where $A$ is the
same abstract $k$-algebra. 

Notice that the system of finite type algebras $(\mathscr{U}, A)$ is filtered in
the following sense. For any two finite type algebras $(\mathscr{U}, A)$ and
$(\mathscr{U}', B)$ at $x \in |\mathscr{X}|$, we may consider the ring
$C \subseteq H^0(\mathscr{U} \cap \mathscr{U}', \cO_{\mathscr{X}})$ generated by the
restrictions of $A$ and $B$ to $\mathscr{U}\cap \mathscr{U}'$. 
We therefore obtain factorizations of schematically dominant morphisms 
\[f_A|_{\mathscr{U} \cap \mathscr{U}'}: \mathscr{U} \cap \mathscr{U}' \xrightarrow{f_{C}}  \Spec(C) \to \Spec(A)\]
\[f_B|_{\mathscr{U} \cap \mathscr{U}'}: \mathscr{U} \cap \mathscr{U}' \xrightarrow{f_{C}} \Spec(C) \to \Spec(B).\]
In particular, we have inclusions of rings of fractions $\Frac(A), \Frac(B) \subseteq \Frac(C)$ for the integral
domains $A,B,C$. Consider the filtered union 
$\bigcup_{(\mathscr{U}, A)} \Frac(A)$
over all finite type algebras at $x$. It turns out that this union stabilizes.
\begin{proposition}
    Let $\mathscr{X}$ be an integral quasi-separated algebraic stack locally of
    finite type over $k$, and let $x \in |\mathscr{X}|$ be a geometric point. Then there is a finite type algebra $(\mathscr{V}, B)$ 
    at $x \in |\mathscr{X}|$ with maximal fraction field, i.e., 
    such that we have $\Frac(B) = \bigcup_{(\mathscr{U}, A)} \Frac(A)$.
\end{proposition}
\begin{proof}
If we denote by $K$ the fraction field of a
finite type integral scheme dominating $\mathscr{X}$, 
then we always have $\Frac(A) \subseteq K$ for all finite type algebras
$(\mathscr{U},A)$ at $x$. 
The $k$-subfields of $K$ are finitely generated extensions of $k$ (see
\cite[Exp. 10, Lem. 1]{serre-albanese}). In particular, the filtered
union $\bigcup_{(\mathscr{U}, A)} \Frac(A)$
is a finitely generated field extension of $k$. All the generators
are contained in $\Frac(B)$ for some finite type algebra $(\mathscr{V},B)$ at $x$
and we conclude $\Frac(B) = \bigcup_{(\mathscr{U}, A)} \Frac(A)$.
\end{proof}

\begin{definition}
    Given an integral quasi-separated algebraic stack $\mathscr{X}$
    locally of finite type over $k$ and a geometric point $x \in |\mathscr{X}|$, 
    the \emph{schematic fraction field} $k_{sch}(\mathscr{X},x)$ of
    $\mathscr{X}$ at $x$ is the unique maximal fraction field $\Frac(A)$ among finite type algebras $(\mathscr{U}, A)$ at $x$. Note that $k_{sch}(\mathscr{X}, x)$ has transcendence degree $\Sdim_x(X)$ over $k$. We say that a morphism $f_A: \mathscr{U} \to \Spec(A)$ is \emph{$x$-birational} if we have $\Frac(A) = k_{sch}(\mathscr{X},x)$.
\end{definition}

\begin{example}
    Suppose that $X$ is a separated algebraic space. 
    If $x\in X(k)$ has schematic dimension equal to $\dim(X)$, then any $x$-birational schematically dominant morphism $f_A: U \to \mathrm{Spec}(A)$ is generically quasi-finite. By Zariski's main theorem \cite[\href{https://stacks.math.columbia.edu/tag/082I}{Tag 082I}]{sp} and the maximality of the fraction field $\Frac(A)$, it follows that $f_A: U \to \Spec(A)$ is a birational morphism.
\end{example}

Suppose that $x \in \mathscr{X}(k)$. Let $f_A: \mathscr{U} \to \Spec(A)$ be an $x$-birational morphism. The unique closed fiber $f_A^{-1}(f_A(x)) \subseteq \mathscr{U}$ containing $x$ is a locally closed substack of $\mathscr{X}$, and so we may take its schematic closure $\overline{f_A^{-1}(f_A(x))} \subseteq \mathscr{X}$,
i.e., the scheme theoretic image of $f_A^{-1}(f_A(x))\to\mathscr{X}$.
By Noetherian induction applied to any fixed quasi-compact open neighborhood of $x$, 
as we range over all finite type
algebras $(\mathscr{U}, A)$ at $x$, there is a unique minimal closed substack 
of the form $\overline{f_A^{-1}(f_A(x))} \subseteq \mathscr{X}$.

\begin{definition}
\label{def: schematic fiber}
    Given an integral quasi-separated algebraic stack $\mathscr{X}$
    locally of finite type over $k$ and a geometric 
    point $x \in \mathscr{X}(k)$, we define the
    \emph{schematic fiber} $\Fib_x(\mathscr{X}) \subseteq \mathscr{X}$ 
    of $\mathscr{X}$ at $x$ 
    as the unique minimal closed substack $\overline{f_A^{-1}(f_A(x))} \subseteq \mathscr{X}$ 
    as we range over all finite type algebras $(\mathscr{U},A)$ at $x$.
\end{definition}

By definition, for a morphism $f_A: \mathscr{U} \to \Spec(A)$, we have that $f_A(\mathscr{U} \cap \Fib_x(\mathscr{X}))$ is a $k$-point of $\Spec(A)$. Therefore, we may think of $\Fib_x(\mathscr{X})$ as consisting roughly of the closure of all the $k$-points of $\mathscr{X}$ that cannot be locally differentiated from $x$ by mapping to a scheme. 

\begin{remark} \label{remark: connectedness of schematic fiber}
    By the minimality of $\Fib_x(\mathscr{X})$, for any open substack $\mathscr{U}\subseteq \mathscr{X}$ containing $x$, we have that the open $\Fib_x(\mathscr{X}) \cap \mathscr{U}$ is schematically dense in $\Fib_x(\mathscr{X})$. In particular, it follows that every irreducible component of $\Fib_x(\mathscr{X})$ contains $x$, and hence $\Fib_x(\mathscr{X})$ is connected.
\end{remark}

\begin{example}
    We have $\Sdim_x(\mathscr{X})=0$ if and only if $\Fib_x(\mathscr{X}) = \mathscr{X}$.
\end{example}

The local invariants we have introduced satisfy some semi-continuity properties.
In particular, they provide a way to canonically stratify an algebraic stack by
studying its local morphisms to schemes.
\begin{proposition}[Semi-continuity of schematic (fiber) dimension] 
\label{prop: stratification by schematic dimension}
    Let $\mathscr{X}$ be an integral quasi-separated algebraic stack locally of finite type over $k$. Then, the schematic dimension function
    \[ x \in \mathscr{X}(k) \mapsto \Sdim_x(\mathscr{X})\]
    is lower semi-continuous. On the other hand, the schematic fiber dimension function
    \[ x \in \mathscr{X}(k) \mapsto \dim(\Fib_x(\mathscr{X}))\]
    is upper semi-continuous.
\end{proposition}
\begin{proof}
    Let $x\in\mathscr{X}(k)$ be a point with 
    $\Sdim_x(\mathscr{X})=n$
    %schematic dimension $n$ 
    and $\dim(\Fib_x(\mathscr{X})) = m$.
    Then there is a finite type algebra $(\mathscr{U},A)$ at $x$ such that $f_A: \mathscr{U} \to \Spec(A)$ satisfies $\dim(\Spec(A))=n$ and $f_A^{-1}(f_A(x)) = \Fib_x(\mathscr{X}) \cap \mathscr{U}$ has dimension $m$ at $x$.
    
    The morphism $f_A: \mathscr{U} \to \Spec(A)$ witnesses that every $k$-point of $\mathscr{U}$ has schematic dimension at least $n$. This shows the lower semi-continuity of the sche\-ma\-tic dimension.

    On the other hand, by upper semi-continuity of the fiber dimension 
    \cite[Prop. 2.6]{osserman-relative-dimension-stacks}
    for the morphism $f_A$, there is an open substack $\mathscr{V} \subseteq \mathscr{U}$
    such that every $k$-point $y \in \mathscr{V}(k)$ satisfies $\dim_y(f_A^{-1}(f_A(y)))
    \leq m$. By considering the restriction $f_A: \mathscr{V} \to \Spec(A)$, 
    we see that for any $k$-point $y \in \mathscr{V}(k)$ we have the estimate
    $\dim(\Fib_{y}(\mathscr{X})) \leq \dim_y(f_A^{-1}(f_A(y))) \leq m$. 
    Hence we obtain the upper semi-continuity of the dimension of the sche\-ma\-tic fiber.
\end{proof}

\begin{corollary} \label{corollary: locus of schematically trivial points is closed}
    Let $\mathscr{X}$ be an integral quasi-separated algebraic stack locally of finite type over $k$. Then the locus of sche\-ma\-ti\-cally trivial points is closed.
\end{corollary}
\begin{proof}
    This is an immediate consequence of 
    \Cref{prop: stratification by schematic dimension}, 
    as the locus of sche\-ma\-ti\-cally trivial points is the subset
    where the schematic dimension is minimal, i.e., $\Sdim_x(\mathscr{X})=0$.
\end{proof}

\begin{example}[Hironaka's example] \label{example: hironaka} Let us recall a
well-known example of Hironaka that yields a proper smooth three-dimensional
algebraic space which is not a scheme (see \cite[Sect. 2]{Hironaka-non-prjoective-example} and \cite[pg. 14]{knutson}). 
Let $P$ be a smooth
projective variety of dimension $3$ equipped with a free involution
$\sigma: P \to P$. Assume that there are two smooth irreducible curves $C, D
\subset P$ intersecting transversally at exactly two points $c,d \in
P(k)$. Furthermore, suppose that $\sigma(C) = D$.

One then defines two
varieties $V_1$ and $V_2$ as follows. The first variety $V_1$ is obtained by
first blowing up $P \setminus \{c\}$ along $C$ and then blowing up the strict
transform of $D$. On the other hand $V_2$ is defined by blowing up $P \setminus
\{d\}$ along $D$ and then blowing up the strict transform of $C$. Both varieties are
isomorphic over the common open preimage $U$ of $P \setminus \{c,d\}$; we
denote by $V$ the scheme obtained by gluing $V_1$ and $V_2$ along $U$. 
 It can be checked that $V$ is a smooth proper variety equipped with an involution such that the induced projection $V \to P$ is birational
and $\mathbb{Z}/2\mathbb{Z}$-equivariant. 

The quotient $V/(\mathbb{Z}/2\mathbb{Z})$ is a smooth proper three-dimensional non-schematic algebraic space
equipped with a birational morphism $V/(\mathbb{Z}/2\mathbb{Z}) \to
P/(\mathbb{Z}/2\mathbb{Z})$ to the smooth projective variety
$P/(\mathbb{Z}/2\mathbb{Z})$. This gives an example of a smooth proper algebraic
space $X = V/(\mathbb{Z}/2\mathbb{Z})$ such that for any point $x \in X(k)$ we have $\Sdim_x(X) = \dim(X) = 3$. 
If we choose the point $x$ to be non-schematic, then we have $\dim\Fib_x(X) \geq 1$ (see \Cref{prop: characterization of algebraic spaces}), and hence
this is an instance where $\Sdim_x(X) + \dim\Fib_x(X) \neq \dim(X)$.
\end{example}

\subsection{Computing schematic fibers}
In this subsection, we provide two ways aiding in the computation of schematic fibers:
the behavior of the schematic fiber under products and an alternative description of the schematic fiber as a base locus or excess intersection similar to the notion of a 
uniform base point.

We begin with the behavior under products.
If one only cares about the underlying reduced structure of the schematic fiber,
then it commutes with products. This is the case for all applications we have in mind
as already the closed points of the schematic fiber carry the information about
the schematicity of a moduli space, 
see \Cref{prop: schematic points for adequate moduli spaces}.

\begin{proposition}[Schematic fibers and products]
\label{prop: schematic fiber and products}
     Let $\mathscr{X}$ and $\mathscr{Y}$ be integral qua\-si-separated
     algebraic stacks locally of finite
     type over $k$. Let $x \in \mathscr{X}(k)$ and $y \in \mathscr{Y}(k)$.
     Then we have $\Fib_{(x,y)}(\mathscr{X}\times\mathscr{Y})_{\red}=
     \Fib_x(\mathscr{X})_{\red}\times \Fib_y(\mathscr{Y})_{\red}$.
\end{proposition}
\begin{proof}
     Note that for any finite type algebra $(\mathscr{U}, A)$ at $x$ 
     and $(\mathscr{V}, B)$ at $y$, 
     we obtain a finite type algebra $(\mathscr{U} \times \mathscr{V}, A \otimes B)$. 
     This immediately shows the inclusion
     $\Fib_x(\mathscr{X})\times \Fib_y(\mathscr{Y})\supseteq
     \Fib_{(x,y)}(\mathscr{X}\times\mathscr{Y})$.

     For the converse, let $(\mathscr{W},C)$ be a finite type algebra
     at $(x,y)$ such that the fiber of the natural morphism
     $\mathscr{W}\to\Spec(C)$ at $(x,y)$ is $\mathscr{W}\cap \Fib_{(x,y)}(\mathscr{X}\times\mathscr{Y})$.
     Let $y'$ be a $k$-point of $Y$ such that $(x,y')$ lies in $\mathscr{W}$.
     Then we define a finite type algebra $(\mathscr{W}_{y'},C_{y'})$
     at $x$ via setting $\mathscr{W}_{y'}$ as the preimage of $\mathscr{W}$
     under $\mathscr{X}\times \{y'\}\to \mathscr{X}\times \mathscr{Y}$
     and $C_{y'}$ as the image of $C$ under the induced morphism
     $H^0(\mathscr{W},\mathcal{O}_{\mathscr{W}})\to 
     H^0(\mathscr{W}_{y'},\mathcal{O}_{\mathscr{W}_{y'}})$.
     
     Consider such a $k$-point $y'$ such that in addition 
     $\Fib_{(x,y)}(\mathscr{X}\times\mathscr{Y})_{y'}$ is non-empty.
     Then we claim
     $\Fib_{(x,y)}(\mathscr{X}\times\mathscr{Y})_{y'} = \Fib_{x}(\mathscr{X})$
     and an analogous claim in for fibers in the $\mathscr{X}$ direction.

     To see the claim consider the commutative diagram
     \[
     \begin{tikzcd}
        \Fib_{(x,y)}(\mathscr{X}\times \mathscr{Y})_{\mid \mathscr{W}} \ar[r] & \mathscr{W} \ar[r] & \mathrm{Spec}(C)\\
        \Fib_{(x,y)}(\mathscr{X}\times \mathscr{Y})_{\mid \mathscr{W}_{y'}} \ar[r] \ar[u]& \mathscr{W}_{y'}\ar[u] \ar[r] & \mathrm{Spec}(C_{y'})\ar[u,hook].
     \end{tikzcd}
     \]
     By the definition of schematic fiber, the composition of the top row factors through 
     the closed point $(x,y)$ of $\Spec(C)$. As $C\twoheadrightarrow C_{y'}$
     by definition of $C_{y'}$, we have a closed immersion on the level of schemes.
     Thus, the composition of the bottom row also factors via the closed point $(x,y)$
     as $\Fib_{(x,y)}(\mathscr{X}\times \mathscr{Y})_{y'}$ is non-empty
     by assumption. We obtain 
     \[
        (\Fib_{(x,y)}(\mathscr{X}\times\mathscr{Y})\cap \mathscr{W})_{y'} = 
        \mathscr{W}_{y'}\times_{\Spec(C)}(x,y) = \mathscr{W}_{y'}\times_{\Spec(C_{y'})}(x,y)
     \]
     which contains the schematic fiber 
     $\Fib_x(\mathscr{X})\cap\mathscr{W}_{y'}$
     by the definition of schematic fiber. We obtain the inclusion
     $\Fib_{(x,y)}(\mathscr{X}\times\mathscr{Y})_{y'}\supseteq \Fib_{x}(\mathscr{X})$
     by taking closures in $\mathscr{X}$.
     The other inclusion holds by the first part of the proof. 
     Thus, we have the equality 
     $\Fib_{(x,y)}(\mathscr{X}\times\mathscr{Y})_{y'} = \Fib_{x}(\mathscr{X})$.
     
     Note that for
     $x$ itself we have $y\in\mathscr{W}_{x}$ and that 
     $\Fib_{(x,y)}(\mathscr{X}\times\mathscr{Y})_x$ is non-empty as it 
     contains $y$. Applying the claim we obtain
     $\Fib_{(x,y)}(\mathscr{X}\times\mathscr{Y})_{x} = \Fib_{y}(\mathscr{Y})$.
     Thus, the assumptions of the claim also hold for $k$-points $y'$
     of $\mathscr{W}_x\cap \Fib_{y}(\mathscr{Y})$.
     By another application we obtain
     $\Fib_{(x,y)}(\mathscr{X}\times\mathscr{Y})_{y'} = \Fib_{x}(\mathscr{X})$
     for such $k$-points $y'$.
     
     We conclude that 
     $\Fib_x(\mathscr{X})\times  (\Fib_y(\mathscr{Y}) \cap \mathscr{W}_x)$ 
     and $\Fib_{(x,y)}(\mathscr{X}\times\mathscr{Y})\cap
     (\mathscr{X}\times\mathscr{W}_x)$ have the same $k$-points. Furthermore,
     both are dense in $\Fib_x(\mathscr{X})\times \Fib_y(\mathscr{Y})$ and
     $\Fib_{(x,y)}(\mathscr{X}\times\mathscr{Y})$ respectively. As the $k$-points
     of an algebraic stack locally of finite type lie dense, we conclude.
\end{proof}

Next, we provide a ``base locus'' description of the schematic fiber in terms of 
certain Weil divisors.
\begin{definition} \label{defn: local weil divisors}
    Let $x \in \mathscr{X}(k)$ be a $k$-point of an integral quasi-separated
    algebraic stack $\mathscr{X}$ locally of finite type over $k$. 
    We denote by $\mathrm{Weil}_x$ the set of codimension 1 closed substacks
    $\mathscr{Z} \subset \mathscr{X}$ that contain $x$, and become principal
    over an open substack $\mathscr{U} \subseteq \mathscr{X}$ containing $x$,
    i.e., over $\mathscr{U}$ they are Cartier and cut out by a section of the
    structure sheaf $\mathcal{O}_{\mathscr{U}}$.
\end{definition}

\begin{remark}
 We note that $\mathrm{Weil}_x$ can be empty; this happens exactly when the point $x$ is schematically trivial (see \Cref{prop:schematic fiber as base locus}), e.g., if $x$ is a uniform base point, see \Cref{lemma: ubp implies schematically trivial}.
\end{remark}

\begin{notation} 
    \label{notn: local components around a point}
    Given a quasi-separated algebraic
    stack $\mathscr{H}$ locally of finite type over $k$ and a $k$-point $x \in \mathscr{H}(k)$, consider the open neighborhood $\mathscr{V} \subseteq \mathscr{H}$ of $x$ obtained by removing all the irreducible components of $\mathscr{H}$ that don't contain $x$. We denote by $\mathrm{L}_x(\mathscr{H})$ the closure of $\mathscr{V}$ in $\mathscr{H}$. It is the union of the irreducible components of $\mathscr{H}$ containing $x$ with appropriate closed substack structure.
\end{notation}

\begin{definition}
    Let $x \in \mathscr{X}(k)$ be a $k$-point of an integral 
    quasi-separated algebraic
    stack $\mathscr{X}$ locally of finite type over $k$. Consider the closed substack $\bigcap_{\mathscr{Z} \in \mathrm{Weil}_x} \mathscr{Z}$ of $\mathscr{X}$. The \emph{local base locus}
    $\mathrm{Bs}_x(\mathscr{X})$ is the closed substack of $\mathscr{X}$ defined by $\mathrm{L}_x\left(\bigcap_{\mathscr{Z} \in \mathrm{Weil}_x} \mathscr{Z}\right)$,
    see \Cref{notn: local components around a point} and \Cref{defn: local weil divisors}. 
\end{definition}
Note that, for all sufficiently small open neighborhoods $\mathscr{U}$ of $x$, 
we have that $\mathrm{Bs}_x(\mathscr{X})$ is the closure of 
$\mathscr{U} \cap \left(\bigcap_{\mathscr{Z} \in \mathrm{Weil}_x} \mathscr{Z}\right)$.

\begin{proposition}[Schematic fiber as a local base locus] \label{prop:schematic fiber as base locus}
    Let $x \in \mathscr{X}(k)$ be a $k$-point of an integral 
    quasi-separated
    algebraic stack $\mathscr{X}$ locally of finite type over $k$. Then, we have
    \[\Fib_x(\mathscr{X}) = \mathrm{Bs}_x(\mathscr{X}).\]
\end{proposition}
\begin{proof}
    Let $\mathscr{U}$ be a small enough neighborhood of $x$ such that
    $\mathrm{Bs}_x(\mathscr{X})$ is the closure of 
    $\left(\bigcap_{\mathscr{Z} \in \mathrm{Weil}_x} \mathscr{Z}\right)\cap \mathscr{U}$.
    After replacing $\mathscr{U}$ by a possibly smaller open neighborhood of $x$, 
    we may assume that there is a
    finite type algebra $(\mathscr{U}, A)$ at $x$ inducing a schematically dominant morphism
    $f_A: \mathscr{U} \to \Spec(A)$ such that we have
    $f_A^{-1}(f_A(x)) = \Fib_x(\mathscr{X}) \cap \mathscr{U}$. 
    We show that $\Fib_x(\mathscr{X}) \cap \mathscr{U} = 
    \left(\bigcap_{\mathscr{Z} \in \mathrm{Weil}_x} \mathscr{Z}\right)\cap \mathscr{U}$. 
    This claim concludes the proof by taking closures of both sides.
    
    We start by proving the inclusion 
    $\Fib_x(\mathscr{X}) \cap \mathscr{U} \subseteq 
    \left(\bigcap_{\mathscr{Z} \in \mathrm{Weil}_x} \mathscr{Z}\right) \cap \mathscr{U}$. 
    It suffices to show that for all $\mathscr{Z} \in \mathrm{Weil}_x$, 
    we have $\Fib_x(\mathscr{X}) \subseteq \mathscr{Z}$. 
    For any $\mathscr{Z} \in \mathrm{Weil}_x$, 
    there is an open substack $\mathscr{V}$ containing $x$ 
    and a schematically dominant morphism $f: \mathscr{V} \to \mathbb{A}^1$
    such that $\mathscr{Z} \cap \mathscr{V} = f^{-1}(0)$. 
    Then by the definition of schematic fiber, 
    we have $\Fib_x(\mathscr{X}) \cap \mathscr{V} \subseteq f^{-1}(0)$. 
    Taking closures we get $\Fib_x(\mathscr{X}) \subseteq \overline{f^{-1}(0)}
    \subseteq \mathscr{Z}$.

    For the other inclusion, notice that the $k$-point $f_A(x)$ in the integral affine scheme $\Spec(A)$ is cut out by finitely many Cartier divisors $f_A(x) = \bigcap_{i \in I} Z_i$, where each $Z_i \subseteq \Spec(A)$ is the vanishing locus of a non-zero function $f_i \in A$. Taking the preimage, we see that 
\begin{equation} \label{equation: local equality fib with base locus}
    \Fib_x(\mathscr{X}) \cap \mathscr{U} = \bigcap_{i \in I} f_A^{-1}(Z_i).
\end{equation}
    By construction, we have $\overline{f_A^{-1}(Z_i)} \in \mathrm{Weil}_{x}$. 
    In particular, we have the inclusion
    $\bigcap_{i \in I} \overline{f_A^{-1}(Z_i)} \supset \bigcap_{ \mathscr{Z}
    \in \mathrm{Weil}_x} \mathscr{Z}$, and hence we may conclude from
    \Cref{equation: local equality fib with base locus}.
\end{proof}

\subsection{Nonschematic points of algebraic spaces}
The goal of this subsection is to provide some characterizations of non-schematic points 
of algebraic spaces in terms of some local geometric features.

\begin{proposition} \label{prop: characterization of algebraic spaces}
    Let $X$ be an integral separated algebraic space locally of finite type over $k$, and let $x \in X(k)$ be a $k$-point. Then, the following are equivalent:
    \begin{enumerate}[(i)]
        \item The point $x$ is a schematic point of $X$.
        \item For any positive dimensional closed subspace $Z \subseteq X$
        containing $x$, there is an open neighborhood $U \subseteq X$ of $x$ and
        a function $g \in H^0(U, \cO_U)$ such that $g|_Z$ is not constant.
        \item The schematic fiber $\Fib_x(X)$ has dimension $0$.
    \end{enumerate}
    If one of the equivalent conditions (i) -- (iii) holds, 
    then the schematic fiber $\Fib_x(X)$ is a reduced point.
\end{proposition}
\begin{proof}
The implication (i) $\Rightarrow$ (ii) can be easily seen by taking an affine
neighborhood $U \subseteq X$ of $x$.

For the implication (ii) $\Rightarrow$ (iii), suppose that (ii) holds and that
$\Fib_x(X)$ has positive dimension. Then, there exists
an open neighborhood $U \subseteq X$ of $x$ and a function $g \in H^0(U, \cO_U)$
such that $g|_{\Fib_x(X) \cap U}$ is not constant.
Since $\Fib_x(X)\cap U$ is connected (\Cref{remark: connectedness of schematic fiber}) 
and locally of finite type over $k$, the induced morphism
$g|_{\Fib_x(X) \cap U}: \Fib_x(X) \cap U \to \mathbb{A}^1$ 
being non-constant means it is schematically dominant. 
In particular, the same holds for $g: U \to \mathbb{A}^1$.
Denote by $A \subseteq H^0(U, \cO_U)$ the ring generated by $g$. 
Then $A$ is isomorphic to the polynomial ring $k[t]$ and 
induces a schematically dominant morphism $f_A: U \to \mathbb{A}^1$ 
such that $\Fib_x(X) \cap U$ does not map to a $k$-point. 
This contradicts the definition of $\Fib_x(X)$.

The implication (iii) $\Rightarrow$ (i) follows by applying Zariski's main theorem \cite[\href{https://stacks.math.columbia.edu/tag/082I}{Tag 082I}]{sp} to an $x$-minimal morphism $f_A: U \to \Spec(A)$.

Finally, note that the reducedness of the schematic fiber is clear if (i) holds.
\end{proof}

\begin{example} \label{example: fulghesu example}
    In \cite[Ex. 2.3]{fulghesu-rational-curves}, there is an example of an
    integral smooth proper %irreducible and reduced finite type 
    algebraic space $X$ of dimension $3$ over $\mathbb{C}$ equipped with a flat morphism $\pi: X \to S$ to a projective surface $S$ such that the fibers are rational curves. The main interest here is that $X$ is not a scheme, but has a morphism $\pi: X \to S$ where all fibers are schemes. In this example we can see that any locally closed integral subspace $Y \hookrightarrow X$ of positive dimension has no schematically trivial points. Indeed, any such $Y$ satisfies one of the following:
    \begin{enumerate}[(1)]
        \item The composition $Y \to X \xrightarrow{\pi} S$ is non-trivial.
        \item $Y$ lies in a fiber of $\pi: X \to S$, and hence it is a scheme of dimension 1.
    \end{enumerate}
    Let $x \in X(k)$ be a point in the non-schematic locus of $X$. By
    \Cref{prop: characterization of algebraic spaces}, the dimension of the
    schematic fiber $\Fib_x(X)$ at $x$ is at least $1$, and it must be contained in a
    fiber of $\pi: X \to S$. In particular, $\Fib_x(X)$ is a scheme of dimension 1.
    \end{example}

\begin{theorem}[Global version] \label{prop: reason why algebraic space at a point global}
    Let $X$ be an integral separated algebraic space locally of finite type over $k$. Suppose that $x \in X(k)$ is a $k$-point of $X$ that is not schematic. Then, one of the following holds:
    \begin{enumerate}[(A)]
        \item There is an effective Cartier divisor $D \subset X$ such that the corresponding line bundle $\cO_X(D)$ is not Zariski-locally trivializable around $x$.
        \item There is a positive dimensional closed subspace $Z \subseteq X$
        containing $x$ such that every effective Cartier divisor $D$ containing
        $x$ also contains $Z$. 
        %In this case, we may take $Z = \Fib_x(X)$.
    \end{enumerate}
\end{theorem}
\begin{proof}
    Suppose that neither (A) nor (B) are satisfied, we shall derive a
    contradiction. Since $x$ is not a schematic point, the schematic fiber
    $\Fib_x(X) \subseteq X$ has positive dimension (\Cref{prop: characterization
    of algebraic spaces}). Since (B) does not hold, there exists an effective
    Cartier divisor $D\subset X$ that contains $x$ but it does not contain
    $\Fib_x(X)$. This Cartier divisor is cut out by a section $s: \cO_X \to
    \cO_X(D)$. Since (A) does not hold, there exists an open
    neighborhood $U \subseteq X$ of $x$ such that $\cO_X(D)$ restricted to $U$ 
    is the trivial line bundle.
    Then $s|_{U}$ is a function in $H^0(U, \cO_U)$ that restricts to a
    non-constant function on $\Fib_x(X)$. The induced schematically 
    dominant morphism $f: U \to \mathbb{A}^1$ does not map $\Fib_x(X) \cap U$ 
    to a $k$-point of $\mathbb{A}^1$, thus contradicting the definition of $\Fib_x(X)$.
\end{proof}

\begin{example}
    Let $X$ be an integral separated algebraic space with a uniform
    base point $x \in X(k)$. Then we have $\Fib_x(X) = X$ (\Cref{thm:main}), and
    therefore (B) in \Cref{prop: reason why algebraic space at a point global}
    cannot hold for any effective Cartier divisor (recall that all effective
    Cartier divisors contain $x$ by the definition of a uniform base point).
    This actually implies that every effective Cartier divisor is not Zariski
    locally trivializable around $x$, so a stronger form of (A) holds. Note that
    this agrees with what we saw in \Cref{corollary: ubp implies line bundles
    are not zariski locally trivializable}.
\end{example}

\begin{remark} \label{remark: schemes do not necessarily satisfy property (B)}
    We note that there are proper schemes over $k$ such that every point satisfies condition (B) in \Cref{prop: reason why algebraic space at a point global}. Indeed, any proper integral scheme with trivial Picard group (such as the one provided by $X_{\Sigma_2}$ in \cite[Ex. 3.5]{eikelberg-picard-grups-toric}) would give such an example.
\end{remark}

    In the case when the algebraic space $X$ is locally factorial, then 
    \Cref{prop: reason why algebraic space at a point global} actually characterizes non-schematic points. 
    \begin{corollary}
         Let $X$ be a separated irreducible locally factorial algebraic space over $k$, and let $x \in X(k)$ be a $k$-point. Then, $x$ is not schematic if and only if at least one of the following holds:
    \begin{enumerate}[(A)]
        \item There is an effective Cartier divisor $D \subset X$ such that the corresponding line bundle $\cO_X(D)$ is not Zariski-locally trivializable around $x$.
        \item There is a positive dimensional closed subspace $Z \subseteq X$ containing $x$ such that every effective Cartier divisor $D$ containing $x$ also contains $Z$.
    \end{enumerate}
    \end{corollary}
    \begin{proof}
        We have seen that a non-schematic point $x$ would satisfy at least one of
        (A) or (B) in \Cref{prop: reason why algebraic space at a point global}.
        For the converse, let $x \in X(k)$ be a schematic point. This means that
        there exists an open neighborhood $U \subseteq X$ of $x$ such that $U$ is
        an affine scheme. Then (A) does not hold, as any Cartier divisor on $U$
        is Zariski locally trivializable. On the other hand, for any closed subspace $Z \subseteq X$ of
        positive dimension containing $x$, we use the affineness of $U$ to find
        a function $f \in H^0(U, \cO_U)$ such that $f$ vanishes at $x$ but does
        not vanish at a point of $Z \cap U$. By the local factoriality of $X$,
        the closure $\overline{\mathbb{V}(f)} \subset X$ of the vanishing locus of $f$ is
        a Cartier divisor, which exhibits the failure of (B).
    \end{proof}

We also formulate a local version of \Cref{prop: reason why algebraic space at a
point global}, which characterizes nonschematic points 
also
in the case when $X$ is not locally factorial and not separated.
\begin{theorem}[Local version]\label{thm: reason why algebraic space at a point local}
    Let $X$ be an integral quasi-separated algebraic space locally of finite
    type over $k$ and let $x \in X(k)$ be a $k$-point. Then, $x$ is not
    schematic if and only if at least one of the following holds:
    \begin{enumerate}[(A)]
        \item There is an open neighborhood $U \subseteq X$ of $x$ and an effective Cartier divisor $D \subset U$ such that the corresponding line bundle $\cO_U(D)$ is not Zariski-locally trivializable around $x$.
        \item There is a positive dimensional closed subspace $Z \subseteq X$ containing $x$ such that for every open neighborhood $U \subseteq X$ of $x$ and every effective Cartier divisor $D \subset U$ containing $x$, we have $Z \cap U \subseteq D$. 
        \item $X$ is not Zariski-locally separated at $x$, i.e., there is no Zariski open neighborhood of $x$ that is separated.
    \end{enumerate}
\end{theorem}
\begin{proof}
    All conditions are local around the point. If $x$ is a schematic point of
    $X$, then it can easily be seen that neither (A), (B), nor (C) could
    possibly hold. On the other hand, if $x$ is not a schematic point of $X$,
    then an argument by contradiction similar to the proof of \Cref{prop: reason
    why algebraic space at a point global} shows that at least one of (A), (B),
    or (C) must be satisfied.
\end{proof}

\begin{remark}
    We note that the conditions (A) and (B) 
    of \Cref{prop: reason why algebraic space at a point global}
    can happen simultaneously.
    Let $X$ be a smooth separated integral
    algebraic space with a positive dimensional
    uniform base locus 
    $|X|_{bs}$. Note that this can not happen in dimension $2$, as the schematic locus is big
    and the uniform base locus lies in the complement.
    For higher dimension such spaces exist, see \Cref{example-kollar}.
    
    Then both (A) and (B) exhibit the unschematicness of closed points in $|X|_{bs}$.
    First note that taking the closure of a reduced effective Cartier divisors 
    defined on the schematic locus gives us an effective Cartier divisor on $X$.
    For (A) recall that only the structure sheaf is Zariski-locally trivializable around $x$,
    see \Cref{corollary: ubp implies line bundles are not zariski locally trivializable},
    so any effective Cartier divisor on $X$ exhibits the unschematicness of $x$ via (A).
    For (B) the uniform base locus is positive dimensional by assumption and
    is contained in every effective Cartier divisor by definition, exhibiting 
    the unschematicness of $x$ via (B).
\end{remark}

\subsection{Schematic points of weak topological moduli spaces}
Due to the mostly
topological nature of our arguments, we are able to obtain results in 
the
following context.
\begin{definition}[Weak topological moduli space] \label{remark: more general contexts}
    Let $\mathscr{X}$ be an algebraic stack and let $X$ be an algebraic space, both of which are locally of finite type over $k$. A morphism $\pi: \mathscr{X} \to X$ is called a \emph{weak topological moduli space} if the following are satisfied:
    \begin{enumerate}[(i)]
        \item $\pi$ is a uniform categorical space (i.e., for all flat morphisms $Y \to X$ from an algebraic space $Y$, the base-change $\pi_Y: \mathscr{X}\times_X Y \to Y$ is a categorical space).
        \item $\pi$ is universally closed, quasi-compact, and  quasi-separated.
        \item For any $k$-point $y \in X(k)$, the stack $\pi^{-1}(y)$ has a unique closed point.
    \end{enumerate}
\end{definition}

The following proposition follows immediately from the definition (cf. \cite[Prop. 5.4.1]{ams}).
\begin{proposition} \label{prop: properties of wtms}
    Let $\pi: \mathscr{X} \to X$ be a weak topological moduli space. Then, the following hold:
    \begin{enumerate}[(a)]
    \item Let $f: Y \to X$ be a flat morphism, where $Y$ is an algebraic space locally of finite type over $k$. Then, the base-change $\pi_Y: \mathscr{X}\times_X Y \to Y$ is a weak topological moduli space.
    \item $\pi$ is surjective.
    \item If $\mathscr{X}$ is irreducible, then $X$ is irreducible.
        \item If $\mathscr{X}$ is reduced, then $X$ is reduced.
        \item If $\mathscr{X}$ is normal, then $X$ is normal.
        \item If $\mathscr{X}$ is quasi-separated, then $X$ is quasi-separated.\qed
    \end{enumerate}
\end{proposition}

\begin{example}
    \quad
    \begin{enumerate}[(1)]
        \item An adequate moduli space (in the sense of \cite[Defn. 5.1.1]{ams}) is a weak topological moduli space.
        \item A topological moduli space (in the sense of Rydh's upcoming work \cite{rydh_tms}) is a weak topological moduli space.
        \item If $\pi: \mathscr{X} \to X$ is a gerbe \cite[\href{https://stacks.math.columbia.edu/tag/06QB}{Tag 06QB}]{sp}, then it is a weak topological moduli space.
    \end{enumerate}
\end{example}

\begin{lemma}[Schematic fiber under moduli space morphisms]\label{lemma: schematic fiber for stack with adequate moduli space}
    Let $\mathscr{X}$ be an integral quasi-separated
    algebraic stack locally of finite type over
    $k$. Suppose that $\mathscr{X}$ admits a weak topological moduli space $\pi:
    \mathscr{X} \to X$. Let $x \in \mathscr{X}(k)$ be
    a closed point. Then we have $\Fib_x(\mathscr{X}) =
    \pi^{-1}(\Fib_{\pi(x)}(X))$ as closed substacks of $\mathscr{X}$.
\end{lemma}
\begin{proof}
The inclusion $\Fib_x(\mathscr{X}) \subseteq \pi^{-1}(\Fib_{\pi(x)}(X))$ is immediate from the definition
of schematic fiber. For the converse, choose a finite type algebra
$(\mathscr{U}, A)$ at $x$. 
Consider the corresponding schematically dominant morphism $f_A: \mathscr{U} \to \Spec(A)$. 
It suffices to show that, after perhaps shrinking the open neighborhood $\mathscr{U}$, 
there is an open subspace $U \subseteq X$ containing $\pi(\mathscr{U})$ and together with a
factorization $f_A: \mathscr{U} \xrightarrow{\pi} U \to \Spec(A)$.

Note that $\mathscr{U}$ contains the fiber $\pi^{-1}(\pi(x))$. 
Indeed, since $\pi^{-1}(\pi(x))$ is quasi-compact and quasi-separated, 
every point in $|\pi^{-1}(\pi(x))|$ specializes to a closed point,
see \cite[\href{https://stacks.math.columbia.edu/tag/0DQN}{Tag 0DQN}]{sp}. 
Since $x$ is the unique point closed point of $\pi^{-1}(\pi(x))$, 
it follows that every point in $|\pi^{-1}(\pi(x))|$ is a generalization of $x$. 
We conclude that $\mathscr{U}$ contains $\pi^{-1}(\pi(x))$, 
because it is closed under generalization and contains $x$.

Consider the closed complement $\mathscr{Z} := \mathscr{X} \setminus \mathscr{U}$ 
(with its reduced substack structure). 
The image $\pi(\mathscr{Z})$ is closed
and does not contain $\pi(x)$. 
Let $U:= X \setminus \pi(\mathscr{Z})$ denote the open complement containing $\pi(x)$. 
Then $\pi^{-1}(U)\subseteq \mathscr{U}$ is an open substack, which has weak topological moduli space $\pi^{-1}(U) \to U$. 
By the universal property of weak topological moduli spaces, 
we obtain the desired factorization $f_A: \pi^{-1}(U) \xrightarrow{\pi} U \to \Spec(A)$.
\end{proof}

\begin{theorem} \label{prop: schematic points for adequate moduli spaces}
    Let $\mathscr{X}$ be an integral 
    quasi-separated
    algebraic stack locally of finite type over $k$. 
    Let $x \in \mathscr{X}(k)$ be a closed point. 
    Suppose that $\mathscr{X}$ admits a separated weak topological moduli space 
    $\pi: \mathscr{X} \to X$. 
    Then the following are equivalent:
\begin{enumerate}[(i)]
        \item The schematic fiber $\Fib_x(\mathscr{X})$ has $x$ as its unique closed point.
        \item We have $\Fib_x(\mathscr{X}) = \pi^{-1}(\pi(x))$.
        \item The image $\pi(x)$ is a schematic point of $X$.
        \item For every open substack $\mathscr{U} \subseteq \mathscr{X}$ containing $x$ that admits a separated weak topological moduli space $\eta: \mathscr{U} \to U$, the image $\eta(x)$ is a schematic point of $U$.
    \end{enumerate}
\end{theorem}
\begin{proof}
    The equivalence (i) $\Leftrightarrow$ (ii) follows from the fact that every fiber of $\pi$ over a $k$-point $y \in X(k)$ contains a unique closed point of $\mathscr{X}$. 
    
    The equivalence (ii) $\Leftrightarrow$ (iii) can be seen as follows. 
    By \Cref{lemma: schematic fiber for stack with adequate moduli space}, we have $\Fib_x(\mathscr{X}) = \pi^{-1}(\pi(x))$ if and only if 
    $\Fib_{\pi(x)}(X) = \pi(x)$. By \Cref{prop: characterization of algebraic spaces}, we have $\Fib_{\pi(x)}(X) = \pi(x)$ if and only if $\pi(x)$ is a schematic point of $X$.
    
    The implication (iv) $\Rightarrow$ (iii) is immediate. For the converse, we equivalently show that (ii) holds for any such open substack $\mathscr{U}$ by noting that $\eta^{-1}(\eta(x)) = \pi^{-1}(\pi(x))$ from the proof of \Cref{lemma: schematic fiber for stack with adequate moduli space} and $\Fib_x(\mathscr{U}) = \Fib_x(\mathscr{X}) \cap \mathscr{U}$ from the definition of schematic fiber. (One may also proceed directly by noting that an argument as in the proof of \Cref{lemma: schematic fiber for stack with adequate moduli space} shows that the image $\pi(x)$ has a common open neighborhood $U \supset V \subseteq X$).
\end{proof}

 Using the terminology of \Cref{defn: local weil divisors}, 
 we may express the existence of sche\-ma\-tic points
 from \Cref{prop: schematic points for adequate moduli spaces} in terms of separating closed points with codimension 1 substacks.
\begin{proposition}
\label{prop: weil-separated and schematic}
    Let $x \in \mathscr{X}(k)$ be a closed point of an integral 
   quasi-separated stack $\mathscr{X}$ locally of finite type over $k$ that admits a separated weak topological moduli space. Suppose that for any other closed point $y \in \mathscr{X}(k)$, there exists a substack $\mathscr{Z} \in \mathrm{Weil}_x$ (see \Cref{defn: local weil divisors}) such that $y \notin \mathscr{Z}$. Then, for every open substack $\mathscr{U} \subseteq \mathscr{X}$ containing $x$ that admits a separated adequate moduli space $\pi: \mathscr{U} \to U$, the image $\pi(x)$ is a schematic point of $U$.
\end{proposition}
\begin{proof}
    This is an immediate consequence of \Cref{prop: schematic points for adequate moduli spaces} and \Cref{prop:schematic fiber as base locus}.
\end{proof}

\section{Applications to moduli of bundles} \label{section: applications to the moduli of bundles}

In this section, we apply the theory developed to the stacks of
$G$-bundles on a smooth projective curve for a connected reductive group $G$.
We show that the semistable loci of $G$-bundles are the maximal open loci admitting
schematic weak topological moduli spaces
via schematic fibers. This generalizes and provides an alternative argument for the analogous
theorem for schematic adequate moduli spaces in \cite[Thm. A]{weissmann-zhang}. 

We also study open substacks $\mathscr{U}$ inside the semistable locus where we show that $\mathscr{U}$ admits a separated schematic weak topological moduli space if and only if the induced morphism to the moduli space of semistable $G$-bundles is an open immersion, \Cref{thm: open substacks of semistable G-bundles}.

Throughout this section, let $C$ be a smooth projective connected curve over an algebraically closed field $k$ and let $G$ be a connected reductive group over $k$. We denote by $\Bun_G$ the stack of principal $G$-bundles on $C$.

\subsection{Moduli of principal bundles on a curve} \label{subsection: moduli principal bundles}

The stack of principal $G$-bundles on $C$
is smooth, see \cite[Prop. 4.1]{hoffmann-moduli-stacks},
and its connected components correspond to elements in $\pi_1(G)$, 
see \cite[Thm. 5.8]{hoffmann-moduli-stacks}.
Principal $G$-bundles in the component $\mathscr{B}un_G^d$
corresponding to $d\in \pi_1(G)$
are said to have \emph{degree} $d$.

\subsubsection{Stability}\label{subsubsection: stability}
We briefly recall the notion of \emph{stability}.
A principal $G$-bundle $\mathcal{P}$ on $C$ is called \emph{semistable} (resp., \emph{stable})
if for all reductions $\mathcal{P}'$ of $\mathcal{P}$ to a parabolic subgroup $H\subseteq G$
and all dominant characters $\chi:H\to \gm$ the associated line bundle $\chi_{\ast}\mathcal{P}'$
has degree $\leq 0$ (resp., $<0$). If $G=\gl_r$, then this notion of semistability coincides
with the notion of slope-semistability.
Semistability is an open property in families. We denote by
$\mathscr{B}un_G^{d,(s)s}\subseteq \Bun_G^d$ the open substack
of (semi)stable principle $G$-bundles of degree $d$.

The stack $\mathscr{B}un_G^{d,ss}$ admits an adequate
moduli space $\mathscr{B}un_G^{d,ss} \to M^{d,ss}_G$. There exists a (possibly empty) open subspace $M_G^{d,s} \subseteq M^{d,ss}_G$ whose preimage is the stable locus $\Bun_{G}^{d,s} \subseteq \mathscr{B}un_G^{d,ss}$. The restriction $\Bun_{G}^{d,s} \to M_G^{d,s}$ is a homeomorphism and all stabilizers of $\Bun_{G}^{d,s}$ have dimension $\dim(Z_G)$, where $Z_G \subseteq G$ is the center. The moduli space $M^{d,ss}_G$ is a normal projective
variety (the construction as a quasi-projective normal variety is done in 
\cite[Cor. 5.5.3]{glss.large}, and projectivity follows from \cite{heinloth-semistable-reduction} and \cite{heinloth-addendum}).
If the genus $g_C$ of $C$ is at least $2$, 
then $M^{d,ss}_G$ has dimension $\dim(G)(g_C-1) + \dim(Z_G)$. 
Indeed,  
by a standard deformation
theory argument using Riemann-Roch, the stack $\Bun_{G}^{d,ss}$ has dimension equal to
$\dim(G)(g_C-1)$, the stable locus $\Bun_{G}^{d,s}$ is non-empty 
(\cite[Prop. 3.25]{holla2002parabolicreductionsprincipalbundles}), 
and we conclude
$\dim(M^{d,s}_G)=\dim(\Bun_{G}^{d,s})+\dim(Z_G)$.

\subsubsection{Determinant of cohomology} \label{subsubsection: determinant of cohomology}
There is a natural
line bundle $\mathcal{L}_{\det}$ on $\mathscr{B}un_G$, given by the \emph{determinant of cohomology}
(see e.g., \cite[1.F.a]{hmheinloth}).
Via the adjoint representation $\mathrm{Ad}: G \to \gl(\mathrm{Lie}(G))$, we
associate to any principal $G$-bundle $\mathcal{P}$ (not necessarily on $C$) its \emph{adjoint bundle} $\mathrm{Ad}_*\mathcal{P}:=\mathcal{P} \times^G \mathrm{Lie}(G)$. Using the universal principal bundle $\mathcal{P}_{univ}$ 
on $C \times \Bun_G$ we 
define
$ \mathcal{L}_{\det} := \det (\mathbb{R}\pr_{2,\ast}
   \mathrm{Ad}_{\ast}\mathcal{P}_{univ})^{-1}$ on $\mathscr{B}un_G$,
where $\pr_2:C\times \Bun_G \to \Bun_G$ denotes the projection. For any principal $G$-bundle $\mathcal{P}$ on $C$ we have $\mathcal{L}_{\det,\mathcal{P}}=
\det H^*(C,\mathrm{Ad}_{\ast}\mathcal{P})^{-1}$. If $G=\gl_n$, then $\mathcal{L}_{\det,\mathcal{E}}=
\det H^*(C,\mathcal{E}nd(\mathcal{E}))^{-1}$
for any vector bundle $\mathcal{E}$ on $C$.

\subsubsection{$\Theta$-semistability}
Associated to any line bundle $\mathcal{L}$ on an algebraic stack $\mathscr{X}$, 
there are several notions of semistable loci of $\mathcal{L}$. 
We briefly recall \emph{$\Theta$-semistability} defined in \cite{hmheinloth} and \cite{hl-theta}.
Let $\mathbb{K}/k$ be an algebraically closed field. A line bundle on the
quotient stack
$\Theta_{\mathbb{K}}:=[\mathbb{A}^1_{\mathbb{K}}/\mathbb{G}_{m,\mathbb{K}}]$
corresponds to a line bundle on $\mathbb{A}^1_{\mathbb{K}}$, which is
necessarily trivial, together with a $\mathbb{G}_{m,\mathbb{K}}$-action, which
is determined by the weight of $\mathbb{G}_{m,\mathbb{K}}$-action on the
(1-dimensional) fiber at $0 \in \mathbb{A}^1_{\mathbb{K}}$. Such an action is
given by a character of $\mathbb{G}_{m,\mathbb{K}}$, so we have an isomorphism
$\wt: \pic(\Theta_{\mathbb{K}}) \xrightarrow{\sim} \mathbb{Z}$.

A geometric point $x\in\mathscr{X}(\mathbb{K})$
is \emph{$\Theta$-semistable} with respect to $\mathcal{L}$ if for
all morphisms $f: \Theta_{\mathbb{K}} \to \mathscr{X}$
such that $f(1)\cong x$ and $f(0)\ncong x$, we have
$\wt(f^{\ast}\mathcal{L})\leq 0$.
We denote the locus of $\Theta$-semistable points of $\mathscr{X}$ by
$\mathscr{X}^{\Theta\text{-}ss}_{\mathcal{L}}$.

In the case of $G$-bundles we have 
$(\Bun^d_{G})^{\Theta\text{-}ss}_{\mathcal{L}_{\det}}=\Bun^{d,ss}_G$,
see \cite[Cor. 1.16]{hmheinloth}.

\subsubsection{Twisted principal bundles}
Let $1\to G \to \hat{G} \xrightarrow{\dt} \mathbb{G}_m\to 1$
be a short exact sequence of connected reductive groups, 
where $G$ is almost simple and simply connected. 
The induced morphism $\dt_{\ast}: \mathscr{B}un_{\hat G} \to \Pic$ 
is smooth by \cite[Cor. 4.2]{hoffmann-moduli-stacks}, 
and we call the image of a $\hat G$-bundle $\hat{\mathcal{P}}$ under $\dt_{\ast}$ the \emph{determinant} of $\hat{\mathcal{P}}$.

For any $d\in\pi_1(\hat G)$, the image $\dt_{\ast}(\mathscr{B}un_{\hat G}^d)$ of the connected component $\mathscr{B}un_{\hat G}^d$ is contained in the connected component of $\Pic$ corresponding to the image of $d\in\pi_1(\hat G)$ under the induced map $\pi_1(\hat G)\to\pi_1(\gm)$. By abuse of notion we still denote the image by $d\in\pi_1(\mathbb{G}_m)$, so $\dt_{\ast}:\mathscr{B}un^{d}_{\hat G}\to\Pic^d$. Restricting to the semistable locus we obtain a morphism of varieties
$\dt: M^{d,ss}_{\hat G}\to \pic^d$ by the universal property of adequate moduli
spaces, see \cite[Thm. 3.12]{Alper2019}.

Let $L$ be a line bundle on $C$ of degree $d$.
There are three canonical constructions we can carry out to fix the determinant:
Consider the fiber product squares
\[
\begin{tikzcd}
  \Bun^{ss}_{\hat G, \cong L}\ar[r] \ar[d] & M^{ss}_{\hat G,L} \ar[d]\ar[r] & \Spec(k) \ar[d,"L",hook]& \ \\
  \Bun^{d,ss}_{\hat G}\ar[r,"\text{ams}"'] & M^{d,ss}_{\hat G} \ar[r,"\dt"'] & \mathrm{Pic}^d
\end{tikzcd}
\]
where the arrow labeled ``ams'' is the adequate moduli space, and
 \begin{equation} \label{diagram: big twisted bundle diagram}
\begin{tikzcd}
    \Bun^{ss}_{\hat G,=L} \ar[r] \ar[d] & \mathscr{B}un_{\hat G, = L}\ar[d]
    \ar[r] & \Spec(k)\ar[d,"\mathbb{G}_m\text{-torsor}"] & \\
    \Bun^{ss}_{\hat G,\cong L} \ar[r] \ar[d] & \mathscr{B}un_{\hat G, \cong L}\ar[d]
    \ar[r] & \rB\gm \ar[r]\ar[d] & \Spec(k)\ar[d,"L",hook] & \\
    \Bun^{d,ss}_{\hat G}\ar[r,hook,"\circ" description] & \Bun_{\hat G}^d \ar[r,"\dt_*"'] & \Pic^d \ar[r,"\mathbb{G}_m\text{-gerbe}"'] &
    \mathrm{Pic}^d.
\end{tikzcd}
\end{equation}
Then we have natural morphisms 
\[
\pi_{=L}: \Bun^{ss}_{\hat G, = L} \xrightarrow{\mathbb{G}_m\text{-torsor}} \Bun^{ss}_{\hat G,\cong L} \xrightarrow{\pi_{\cong L}} M^{ss}_{\hat G,L}.
\]
The difference between $\Bun_{\hat G,\cong L}$ and $\Bun_{\hat G, = L}$
is a subtle one: the former one classifies principal
bundles with determinant being isomorphic to $L$, whereas the latter one has the isomorphism as part of datum.
This difference is invisible on the underlying topological spaces.

We recall for later use that both algebraic stacks $\Bun_{\hat G,\cong L}$ and $\Bun_{\hat G, = L}$ are smooth
and connected. Smoothness follows
from the smoothness of the morphism 
$\dt_{\ast}: \Bun_{\hat G}\to \Pic$,
whereas the connectedness can for example be seen
via global sections $H^0(\Bun_{\hat G,=L},\mathcal{O}_{\Bun_{\hat G,=L}})=k$,
see \cite[(5) just before Prop. 4.2.3]{biswas-hoffmann}.

The following lemma is a generalization of \cite[Prop. 2.1]{hoffmann-pic}, where the proof uses GIT-techniques, whereas here
we provide a proof using the language of adequate moduli spaces.

\begin{lemma}
    \label{lemma: moduli fixing det agree}
    Both $\pi_{=L}$ and $\pi_{\cong L}$ are adequate moduli spaces.
\end{lemma}

\begin{proof}
We first show the claim for $\pi_{\cong L}$.

As the connected component $Z^0$ of the center $Z$ of $\hat G$ is 
isomorphic to $\gm$, the action
$Z^0\times \hat G\to \hat G$ mapping $(z,\hat g) \mapsto z\cdot \hat g$ induces a morphism 
\[
\psi:\Pic^0\times \Bun^{d}_{\hat G}\to\Bun^{d}_{\hat G} \text{ mapping }(\cM,\hat{\mathcal{P}})\mapsto \cM\otimes \hat{\mathcal{P}},
\]
and $\dt_{\ast}(\cM\otimes \hat{\mathcal{P}})=\dt_{\ast}(\hat{\mathcal{P}}) \otimes \cM^{\otimes N}$ for
some integer $N$ (which depends only on $\hat G$). 
Observe that $\psi$ also induces a morphism on the semistable loci.
We also obtain an action after rigidifying by central automorphisms,
i.e.,
\[
\bar{\psi}: \pic^0\times(\Bun^{d,ss}_{\hat G})^{\rig} \to (\Bun^{d,ss}_{\hat G})^{\rig} 
\]
where $\Bun^{d,ss}_{\hat G} \to (\Bun^{d,ss}_{\hat G})^{\rig}$ is the rigidification by the central $Z^0$-auto\-mor\-phisms, induced by $Z^0\to \hat G$. Both $\Bun^{d,ss}_{\hat G}$ and $(\Bun^{d,ss}_{\hat G})^{\rig}$
admit the same adequate moduli space $M^{d,ss}_{\hat G}$ as the rigidification is a $Z^0$-gerbe, see \cite[Thm. 5.1.5 and its proof]{acv_twisted_bundles}.

Taking determinant induces a morphism $(\Bun^{d,ss}_{\hat G})^{\rig}\to \pic^d$
and $(\Bun^{ss}_{\hat G,\cong L})^{\rig}$ is the fiber over $L \in \pic^d(k)$.
Thus, $(\Bun^{ss}_{\hat G,\cong L})^{\rig}$ is the rigidification
of $\Bun^{ss}_{\hat G,\cong L}$ by $Z^0$ and both admit the same adequate moduli
space $X$. Consider the commutative diagram
\[
\begin{tikzcd}
    \mathrm{Pic}^0\times (\Bun^{ss}_{\hat G,\cong L})^{\rig} \ar[d,"\text{ams}"'] \ar[rr,"\bar{\psi}"] & \ &
    (\Bun^{d,ss}_{\hat G})^{\rig}\ar[d,"\text{ams}"]\\
    \mathrm{Pic}^0 \times X  \ar[rr,"\text{by univ. prop. of ams}"] \ar[d,"\pr_1"'] & \ &
    M^{d,ss}_{\hat G}\ar[d,"\dt_{\ast}"]\\
    \mathrm{Pic}^0 \ar[r,"(-)^{\otimes N}"] & 
    \mathrm{Pic}^0 \ar[r,"- \otimes L"] & \mathrm{Pic}^d,
\end{tikzcd}
\]

We claim that the outer square is Cartesian. Then so is the bottom square
since the bottom morphism $\pic^0\to\pic^d$ is flat and the adequate moduli space is stable under flat base change \cite[Prop. 5.2.9 (1)]{ams}. 
A further base change along $\Spec(k)\xrightarrow{\mathcal{O}_C}\pic^0$
then shows that $X\cong M^{ss}_{\hat G,L}$.

To see that the outer square is Cartesian we construct a morphism
\begin{align*}
    \pic^0\times (\Bun^{ss}_{\hat G,\cong L})^{\rig}&\to \pic^0\times_{\pic^d}(\Bun^{d,ss}_{\hat G})^{\rig} \\
    (\cM,\hat{\mathcal{P}})&\mapsto (\cM,\cM\otimes \hat{\mathcal{P}})
\end{align*}
and it admits an inverse given by $(\cM,\hat{\mathcal{P}})\mapsto (\cM,\cM^{-1}\otimes \hat{\mathcal{P}})$.
This concludes the proof of $\pi_{\cong L}$
being an adequate moduli space.

The claim for $\pi_{=L}$ can be derived from this as follows.
By construction, the morphism $\Bun^{ss}_{\hat G, = L}\to \Bun^{ss}_{\hat
G,\cong L}$ is affine as it is a $\gm$-torsor. 
Thus, the composition
$\pi_{=L}:\Bun^{ss}_{\hat G, = L}\to M^{ss}_{\hat G,L}$ 
is adequately affine and the relative spectrum 
\[
\Bun^{ss}_{\hat G,=L}\to X':=\mathrm{Spec}\left(\pi_{=L,\ast}\mathcal{O}_{\Bun^{ss}_{\hat G,=L}}\right)
\]
is an
adequate moduli space \cite[Rmk. 5.1.2]{ams}. Note that $X'$ is normal and irreducible
as the same hold for $\Bun^{ss}_{\hat G,=L}$ \cite[Prop. 5.4.1]{ams}.

It is readily checked that 
on the underlying topological spaces we have an isomorphism
$|\Bun^{ss}_{\hat G, = L}|\xrightarrow{\sim} |\Bun^{ss}_{\hat G,\cong L}|$.
As an adequate moduli space has a unique closed 
point in each fiber \cite[Thm. 5.3.1]{ams}, we conclude that
$X'\to M^{ss}_{\hat G,L}$ is an affine morphism of normal varieties 
which induces a bijection on closed points.
By Zariski's main theorem we conclude that $X'\to M^{ss}_{\hat G,L}$
is an isomorphism.
 \end{proof}

We also have a twisted version of the determinant of cohomology.

\begin{definition}
    Let $1\to G\to \hat G\to \gm\to 1$ be a short exact sequence of connected reductive groups.
    Let $L$ be a line bundle on $C$. Then the \emph{determinant of cohomology} $\mathcal{L}_{\det,L}$
    on the stack $\Bun_{\hat G,=L}$ is defined as the pullback
    of $\mathcal{L}_{\det}$ on $\Bun_{\hat G}$ via $\Bun_{\hat G,=L}\to\Bun_{\hat G}$.
 \end{definition}

We recall that by \cite[Lem. 2.11]{weissmann-zhang} and the equality $\Bun_{\hat G}^{ss}=(\Bun_{\hat G})^{\Theta\text{-}ss}_{\mathcal{L}_{\det}}$ in the untwisted case, we have 
\[
\Bun_{\hat G,=L}^{ss}=(\Bun_{\hat G,=L})^{\Theta\text{-}ss}_{\mathcal{L}_{\det,L}},
\]
where
we defined $\Bun_{\hat G,=L}^{ss}$ as the pullback of the
semistable locus $\Bun_{\hat G}^{ss}$ (see \Cref{diagram: big twisted bundle
diagram}).
\begin{lemma}\label{lemma:app-Bun}
    Let $1\to G \to \hat{G} \xrightarrow{\dt} \mathbb{G}_m\to 1$
    be a short exact sequence of connected reductive groups, where $G$ is almost simple and simply connected.
    Let $L$ be a line bundle on $C$. Suppose $g_C\geq 2$. Then any unstable $\hat{G}$-bundle $\hat{\mathcal{P}}$ on $C$ with determinant $L$ is a uniform base point
    of $\Bun_{\hat G,=L}$
    and we have 
    \[
    \Fib_{\hat{\mathcal{P}}}(\Bun_{\hat G,=L})=\Bun_{\hat G,=L}.
    \]
\end{lemma}
\begin{proof}

    Let $\hat{\mathcal{P}}$ be an unstable $\hat G$-bundle with determinant $L$.
    If $\hat{\mathcal{P}}$ is a uniform base point of $\Bun_{\hat G,=L}$, then its schematic fiber is the entire stack by 
    \Cref{prop:schematic fiber as base locus} as every Zariski open neighborhoods of a uniform base point is big.% by \Cref{lemma: zariski neighborhoods around ubp}.
    
    Since $\Bun_{\hat G,=L}$ is smooth, every prime divisor on $\Bun_{\hat G,=L}$
    is the vanishing locus of a global section of a line bundle on $\Bun_{\hat G,=L}$. 
    Recall that the non-vanishing locus of a global section of a line bundle is contained in its $\Theta$-semistable locus (see \cite[Lem. 2.4, the proof of (ii) $\Rightarrow$ (iii)]{weissmann-zhang}). We know the $\Theta$-semistable loci of 
    $\mathcal{L}_{\det,L}$ and its inverse (see \cite[Subsect. 2.2.1 and Lem. 2.11]{weissmann-zhang}):
    \[(\Bun_{\hat G,=L})^{\Theta\text{-}ss}_{\mathcal{L}_{\det,L}^{\otimes e}} =
    \begin{cases}
          \emptyset &\text{ if }e < 0,\\
          \Bun_{\hat G,=L} &\text{ if }e=0,\\
          \Bun_{\hat G,=L}^{ss} &\text{ if }e>0.
    \end{cases}
    \]
    By \cite[Cor. 4.4.5, Thm. 4.2.1 (i), and (5) just before Prop. 4.2.3]{biswas-hoffmann}
    we have $\pic(\Bun_{\hat{G},=L}) \cong \mathbb{Z}$
    and $H^0(\Bun_{\hat G,=L},\mathcal{O}_{\Bun_{\hat G,=L}})=k$.
    Thus, up to positive powers only $\mathcal{L}_{\det,L}$ can possibly admit global sections with non-trivial vanishing locus.
    In particular, the vanishing locus of a global section of a line bundle $\mathcal{L}$
    on $\Bun_{\hat G,=L}$ defines a prime divisor only if $\mathcal{L}^{\otimes N}\cong \mathcal{L}_{\det,L}^{\otimes e}$
    for some positive integers $N,e>0$.
    Then we conclude that the complement of 
    every prime divisors is contained in $\Bun_{\hat G,=L}^{ss}$, and an unstable $\hat G$-bundle is
    contained in every prime divisor.
\end{proof}

Unschematicness of a moduli space at a closed point can be checked via the schematic
fiber by showing that it contains more than one closed point, see Theorem \ref{prop: schematic points for adequate moduli spaces}.
At an unstable bundle in $\Bun_{\hat G,=L}$ the schematic fiber is the entire stack by \Cref{lemma:app-Bun}.
While $\Bun_{\hat G,=L}$ is not quasi-compact and has no closed point,
closed points are abundant for its quasi-compact opens:
\begin{lemma}
\label{lemma: infinitely many closed points}
Let $1\to G\to \hat G\to \gm\to 1$ be a short exact sequence of connected reductive groups. 
Let $L$ be a line bundle on $C$.
Let $\mathscr{U}\subseteq \Bun_{\hat G,=L}$ be a quasi-compact open substack. 
Suppose $g_C \geq 2$. Then $\mathscr{U}$ has infinitely many closed points.
\end{lemma}

\begin{proof}
We recall the Harder-Narasimhan weak $\Theta$-stratification in the following
(cf. \cite{gurjar2020hardernarasimhanstacksprincipalbundles} or \cite{gaugedmaps}). 
This amounts to a morphism $\varphi: \sqcup_{i \in I} \mathscr{S}_i
\to \Bun_{\hat{G}}$, where each $\mathscr{S}_i$ is an algebraic stack which can be
described as follows.

Let $P \subseteq \hat G$ be a parabolic subgroup with Levi quotient $P \twoheadrightarrow M$. The induced morphism
$\Bun_{P} \to \Bun_{M}$ is quasi-compact with connected fibers (see the proof of
\cite[Prop. 5.1]{hoffmann-moduli-stacks}), and a choice of splitting $M \to P$
induces a section on the level of stacks $\Bun_M \to \Bun_P$. 
For a given degree $d \in \pi_1(M)$, we
denote by $\Bun_P^{d,ss} \subseteq \Bun_P$ the preimage of the open substack
$\Bun_M^{d,ss} \subseteq \Bun_{M}$ of semistable $M$-bundles of degree $d$ on $C$ 
under the morphism $\Bun_{P} \to \Bun_{M}$.
Then each $\mathscr{S}_i$ is of the form $\Bun_{P_i}^{d_i,ss}$ for some
parabolic $P_i \subseteq \hat G$ with Levi quotient $M_i$ and some degree $d_i \in
\pi_1(M_i)$. In particular, $\mathscr{S}_i$ is smooth and integral, and the
corresponding morphism $\mathscr{S}_i \to \Bun_{M_i}^{d,ss}$ is of finite type and surjective. The corresponding morphism $\varphi: \sqcup_{i \in I} \mathscr{S}_i \to \Bun_{\hat{G}}$ induced by extension of structure group to $\hat G$ satisfies the following:
\begin{enumerate}[(a)]
    \item For each index $i \in I$, the image $\varphi(|\mathscr{S}_i|) \subseteq |\Bun_{\hat G}|$ at the level of topological spaces is locally closed. We denote by $\varphi(\mathscr{S}_i) \subseteq \Bun_{\hat G}$ the corresponding reduced locally closed substack.
    \item The induced morphism $\varphi_i: \mathscr{S}_i \to \varphi(\mathscr{S}_i)$ is finite and radicial.
    \item There is an ordering on the set $I$ such that $|\Bun_{\hat G}| = \sqcup_{i \in I} |\varphi(\mathscr{S}_i)|$ is a stratification by locally closed subsets.
\end{enumerate}

Given a line bundle $L$ on $C$, we may base-change each stack $\mathscr{S}_i$
under the morphism $\Bun_{\hat G, = L} \to \Bun_{\hat G}$ in order to get a
smooth integral algebraic stack $\mathscr{S}_{i , =L}$ equipped with a
surjective morphism $\mathscr{S}_{i, = L} \to \Bun_{M_i, =L}^{d_i, ss}$.
Again, the corresponding morphism 
$\varphi: \sqcup_{i \in I} \mathscr{S}_{i, = L} \to \Bun_{\hat{G}, =L}$ 
induced by extension of structure group satisfies
properties (a), (b), and (c) as above.
Note that this notation coincides with the notation on fixing the determinant
introduced earlier, since the composition
$P\to \hat G\to \gm$ for parabolic subgroup factors via its Levi quotient.

We are now ready to proceed with the proof. Let $\mathscr{U} \subseteq \Bun_{\hat
G, =L}$ be a quasi-compact open substack. Then there exists a stratum
$\mathscr{S}_{i, = L}$ such that $\varphi(\mathscr{S}_{i, =L}) \cap \mathscr{U}$
is closed in $\mathscr{U}$. 
Consider the open preimage
$\varphi_i^{-1}(\mathscr{U}) \subseteq \mathscr{S}_{i, =L}$. It follows from the
properties above that the morphism $\varphi_i: \varphi^{-1}_i(\mathscr{U}) \to
\varphi(\mathscr{S}_{i, =L}) \cap \mathscr{U} \hookrightarrow \mathscr{U}$ is
finite and radicial. Therefore, it suffices to show that the stack
$\varphi^{-1}_i(\mathscr{U})$ has infinitely many closed points. 

Since the
morphism $\mathscr{S}_{i, =L} \to \Bun_{M_i, =L}^{d_i, ss}$ is of finite type and
surjective with integral target, it follows by Chevalley's theorem on
constructibility that the image of $\varphi_i^{-1}(\mathscr{U}) \to \Bun_{M_i,
=L}^{ss}$ contains a non-empty open substack $\mathscr{V} \subseteq \Bun_{M_i,
=L}^{d_i,ss}$. 

It suffices to show that $\mathscr{V}$ contains infinitely many
$k$-points that are closed in $\Bun_{M_i, =L}^{d_i,ss}$. Indeed, for each such closed
point $v \in \mathscr{V}(k)$ there is a closed point of the non-empty
quasi-compact and quasi-separated fiber $\varphi_i^{-1}(\mathscr{U})_v$ which
yields a closed point of $\varphi_i^{-1}(\mathscr{U})$. 

Note that the
intersection $\mathscr{V} \cap \Bun_{M_i, =L}^{d_i,s}$ with the stable locus is
an open substack of $\Bun_{M_i, =L}^{d_i, ss}$ such that every $k$-point of
$\mathscr{V} \cap \Bun_{M_i, =L}^s$ is closed in $\Bun_{M_i, =L}^{d_i, ss}$. Hence,
it suffices to show that $\mathscr{V} \cap \Bun_{M_i, =L}^{d_i,s}$ has infinitely many
$k$-points. 

First, we note that $\Bun_{M_i, =L}^{d_i,s}$ is non-empty from our
assumption that the genus is at least $2$ (this is proven in \cite[Prop.
3.5]{holla2002parabolicreductionsprincipalbundles} for every component of the
stack $\Bun_{\hat G}$, which implies the non-emptiness for all fibers $\Bun_{\hat
G, =L}$ of $\Bun_{\hat G} \to \Pic$ by translating by elements of
$\Pic^0$). Since $\mathscr{V} \cap \Bun_{M_i, =L}^{d_i,s}$ is an open dense
substack of the integral stack $\Bun_{M_i, =L}^{d_i, s}$ which admits an adequate
moduli space, it suffices to show that $\Bun_{M_i, =L}^{d_i, s}$ has infinitely many
closed points. 
 
As an adequate moduli space morphism has a unique closed point in a fiber over a closed point,
it suffices to show that the corresponding moduli space
$M^{d_i,s}_{M_i,=L}$ has infinitely many closed points. In this case, the moduli space has
dimension $\dim(\Bun_{M_i, =L}^{d_i, s}) + \dim(Z_{M_i}) -1$,
because the moduli space
of stable bundles is a homeomorphism and the stabilizers of $\Bun_{M_i,
=L}^{d_i, s}$ all have dimension $\dim(Z_{M_i})-1$. 
Hence, by a standard deformation theory argument, $M^{d_i,s}_{M_i,=L}$
is a variety of positive dimension $(\dim(M_i)-1)(g_C-1) + \dim(Z_{M_i})-1$, 
which implies that it has infinitely many points.
\end{proof}

\subsubsection{A reduction lemma}
To study the schematic fiber at an unstable bundle in $\Bun_{G}$ for a general reductive group $G$,
we reduce it to a simpler
setting using the structure theory of reductive groups. We need to study the behavior of the schematic
fiber under the morphism 
$\varphi:\Bun_{G'}^{d'}\to \Bun_{G}^d$ 
induced by a central isogeny $G'\twoheadrightarrow G$ with kernel $\mu$.
By \cite[Ex. 5.1.4]{biswas-hoffmann} the morphism $\varphi$ is a $\Bun_{\mu}$-torsor
in the sense of \cite[Def. 5.1.3]{biswas-hoffmann}.

\begin{lemma}
\label{lemma: Bun_mu torsor finite fibers and universally closed}
    Let $\mathscr{X}$ be an algebraic stack locally of finite type over $k$.
    Further, let $\varphi:\mathscr{Y}\to\mathscr{X}$ be a $\Bun_{\mu}$-torsor, 
    where $\mu$ is a finite multiplicative group.
    Then $\varphi$ has finite fibers and is universally closed.

    If $\mathscr{X}$ admits an adequate moduli space $\mathscr{X} \to X$,
    then $\mathscr{Y}$ admits an adequate moduli space $\mathscr{Y} \to Y$
    and the induced morphism $Y\to X$ is finite.
\end{lemma}
\begin{proof}
    Both properties can be checked after an fppf base change. 
    As a $\Bun_{\mu}$-torsor is fppf by definition, 
    it suffices to check finite fibers and universally closed for the base change 
    $\mathscr{Y}\times_{\mathscr{X}}\mathscr{Y}\cong 
    \mathscr{Y}\times \Bun_{\mu}$.
    %\xrightarrow{\pr_{\mathscr{Y}}}\mathscr{Y}$.
    As $\Bun_{\mu}\to \Spec(k)$ has the desired properties,
    so does the base change 
    $\pr_{\mathscr{Y}}: \mathscr{Y}\times \Bun_{\mu}\to \mathscr{Y}$.

    Assume that $\mathscr{X}$ admits an adequate moduli space $\pi:\mathscr{X}\to X$.
    As adequate affineness can be checked fppf locally, we find that
    $\psi:\mathscr{Y}\to Y:=\Spec((\pi \circ \varphi)_{\ast}\mathcal{O}_{\mathscr{Y}})$ 
    is an adequate moduli space, 
    for more details see \cite[Lem. 2.7]{weissmann-zhang}.
    Then $Y\to X$ is affine by definition.
    
    To show it is finite, it remains to show that it is universally closed
    \cite[\href{https://stacks.math.columbia.edu/tag/04NZ}{Tag 04NZ}]{sp}
    as $Y$ and $X$ are locally of finite type, see \cite[Thm. 6.3.3]{ams}
    (applied to an open covering where each open is quasi-compact),
    and an affine morphism is separated.
    This follows directly from a diagram chase 
    and the fact that $\mathscr{Y}\to Y$ is surjective and universally closed
    and that $\mathscr{Y}\to\mathscr{X}\to X$ is universally closed.
\end{proof}

Such a $\Bun_{\mu}$-torsor appears in the reduction of $\Bun_G^d$ to a product of twisted almost simple
simply connected $\Bun_{\hat G_i,=L}$ and $\Bun_T$ for a torus $T$.
We make use of this and it is helpful to fix the notation.
This reduction set-up was also used in the proof of \cite[Thm. A]{weissmann-zhang}
and goes back to \cite{biswas-hoffmann}.

\begin{situation}
\label{situation: reduction lemma}
    Let $G$ be a connected reductive group and $d\in\pi_1(G)$. Let $Z^0$ be the connected
    component of the center $Z$ of $G$ and let $\tilde G\twoheadrightarrow [G,G]$
    be the universal cover of the derived subgroup of $G$. 

    By \cite[Lem. 5.3.2]{biswas-hoffmann} and the discussion following it,
    there exists an extension 
    \[
    1\to \tilde G \to \hat G\xrightarrow{\dt}\gm \to 1
    \]
    together with an extension $\hat\psi:\hat G\to G$ of $\tilde G\to G$ mapping $1\in\pi_1(\hat G)\cong \mathbb{Z}$ 
    to $d\in\pi_1(G)$.
    The morphism
    $Z^0\times \hat G\to G\times \gm$ given by $(\zeta, \hat g)\mapsto (\zeta \cdot \hat
    \psi(\hat g),\dt(\hat g))$ is a central isogeny, say with kernel $\mu$,
    and the induced morphism
    \begin{equation}
        \label{equation: torsor}
        \varphi:\Bun_{Z^0}^0\times \Bun_{\hat G}^{1}\to \Bun^d_G \times \Pic^1
    \end{equation}
    is a $\Bun_{\mu}$-torsor. Pulling back along a fixed line bundle $\Spec(k)\xrightarrow{L} \Pic^1$
    we obtain a $\Bun_{\mu}$-torsor
    \[
    \varphi_L: \mathscr{X}_L:=\Bun_{Z^0}^0\times \Bun_{\hat G,=L}\to \Bun^d_G.
    \]
    Let $G_1,\dots, G_n$ be the almost simple
    simply connected factors of $\tilde G$.
    Then by \cite[p. 57, (11)]{biswas-hoffmann} 
    there exist extensions $1\to G_i\to\hat G_i\to \gm\to 1$ and
    we have $\Bun_{\hat G,=L}\cong\prod_{i=1}^n \Bun_{\hat G_i,=L}$.
    Setting $T=Z_0$ we obtain a $\Bun_{\mu}$-torsor
    \begin{equation}
        \label{equation: twisted torsor}
        \varphi_L:\mathscr{X}_L:=\Bun^0_{T}\times \prod_{i=1}^n \Bun_{\hat G_i,=L}\to \Bun^d_G.
    \end{equation}    
We note that, since $\varphi$ is induced by a central
    isogeny,
    for any point $x=(\mathcal{M},\hat{\mathcal{P}}) \in \mathscr{X}$, the bundle $\hat{\mathcal{P}}$ is semistable if and only if 
    $\varphi(x)$ is semistable. The same holds for
    any point $x=(\mathcal{M},\hat{\mathcal{P}}_1,\dots,\hat{\mathcal{P}}_n) \in \mathscr{X}_L$:
    the bundles 
    $\hat{\mathcal{P}}_1,\dots,\hat{\mathcal{P}}_n$ are semistable if and only if $\varphi_L(x)$ is semistable.
\end{situation}

\subsection{Semistability and schematic moduli spaces}

We obtain a more robust proof for the semistable locus being canonical
in the language of the schematic fiber which
strengthens the result \cite[Thm. A]{weissmann-zhang}
to include weak topological moduli spaces in the sense of \Cref{remark: more general contexts}.
\begin{theorem}
    \label{thm: semistable and schematic}
    Let $C$ be a smooth projective connected curve of genus $g_C \geq 2$. 
    Let $G$ be a connected reductive group.
    Then the semistable locus $\mathscr{B}un_G^{ss}$ is the unique maximal open substack of $\mathscr{B}un_G$ that admits a schematic weak topological moduli space.
\end{theorem}

\begin{proof}
Let $d\in\pi_1(G)$.
Consider the $\Bun_{\mu}$-torsor 
\[
    \varphi_L:\Bun^0_{T}\times \prod_{i=1}^n \Bun_{\hat G_i,=L}\to \Bun^d_G,
\]
where $\mu$ is a finite multiplicative group,
see \Cref{equation: twisted torsor}.

Let $\mathscr{U}\subseteq \Bun^d_G$ be an open quasi-compact substack containing
an unstable bundle $\mathcal{P}$. As semistability is an open condition
and quasi-compact algebraic stacks of finite type admit closed $k$-points,
we can assume $\mathcal{P}$ to be a closed $k$-point.
Note that $\varphi_L^{-1}(\mathscr{U})\to \mathscr{U}$ is
a $\Bun_{\mu}$-torsor and $\varphi_L^{-1}(\mathscr{U})$ is a quasi-compact open
substack containing a closed $k$-point
point $(\mathcal{M}, \hat{\mathcal{P}}_1,\dots, \hat{\mathcal{P}}_n)$ 
mapping to $\mathcal{P}$ with at least one of the $\hat{\mathcal{P}}_i$ being unstable.

We compute the schematic fiber of such a closed $k$-point 
$(\mathcal{M}, \hat{\mathcal{P}}_1,\dots, \hat{\mathcal{P}}_n)$,
where we order the product in such a way that 
$\hat{\mathcal{P}}_1,\dots,\hat{\mathcal{P}}_m$ are unstable and $\hat{\mathcal{P}}_{m+1},\dots, \hat{\mathcal{P}}_n$
are semistable.
As the schematic fiber commutes with products (if we equip
them with the reduced substack structure), see \Cref{prop: schematic fiber and products}, 
we have
\begin{align*}
    \mathscr{Z}:&=\bigg{(}\Fib_{(\mathcal{M}, \hat{\mathcal{P}}_1,\dots, \hat{\mathcal{P}}_n)}
    (\Bun_T \times \prod_{i=1}^n\Bun_{\hat G_i,=L}) \bigg{)}_{\red} \\
    &=\rB\mathbb{G}_m^{\oplus \rk(T)} \times \prod_{i=1}^m\Bun_{\hat G_i,=L}
    \times  \prod_{i=m+1}^n \pi_{i,=L}^{-1}([\hat{\mathcal{P}}_i])_{\red} 
\end{align*}
by \Cref{lemma: schematic fiber for stack with adequate moduli space}
and \Cref{lemma:app-Bun},
where $\pi_{i,=L}:\Bun^{ss}_{\hat G_i,=L}\to M^{ss}_{\hat G_i,L}$
denotes the adequate moduli space and $[\hat{\mathcal{P}}_i]$ the 
image of $\hat{\mathcal{P}}_i$ under $\pi_{i,=L}$. 
By applying \Cref{lemma: infinitely many closed points}
to the quasi-compact open $\pr_1(\mathscr{Z}\cap \varphi_L^{-1}(\mathscr{U})) =
\pr_1(\varphi_L^{-1}(\mathscr{U}))\subseteq \Bun_{\hat G_1,=L}$
we see that the schematic fiber $\mathscr{Z}\cap \varphi_L^{-1}(\mathscr{U})$
contains infinitely many closed points.

Note that restricting $\varphi_L$ yields a morphism on the level of schematic fibers
$\Fib_{(\mathcal{M}, \hat{\mathcal{P}}_1,\dots, \hat{\mathcal{P}}_n)}(\varphi_L^{-1}(\mathscr{U}))
 \to \Fib_{\mathcal{P}}(\mathscr{U})$. 
As $\varphi_L:\varphi_L^{-1}(\mathscr{U})\to\mathscr{U}$ 
has finite fibers and maps closed points to closed points, 
see \Cref{lemma: Bun_mu torsor finite fibers and universally closed}, 
we find that the schematic fiber
$\Fib_{\mathcal{P}}(\mathscr{U})$ at the unstable bundle $\mathcal{P}$ contains
infinitely many closed points. If $\mathscr{U}$ admits a schematic weak
topological moduli space $\pi: \mathscr{U} \to U$, then replacing $\mathscr{U}$
by the preimage of an affine open neighborhood of $\pi(\mathcal{P})$ we may
assume that $U$ is affine, in particular separated. 
This contradicts \Cref{prop: schematic points for adequate moduli spaces}.
\end{proof}

\subsection{Weak topological moduli spaces of principal bundles}

We need the following
analog of \cite[Defn. 2.5]{edidin-rydh-canonical} in the setting of weak topological moduli spaces.
\begin{definition}
Let $\pi: \mathscr{X} \to X$ be a weak topological moduli space. 
A point $x$ of $\mathscr{X}$ is \emph{stable} relative to $\pi$ if
$|\pi^{-1}(\pi(x))|=\{x\}$ under the induced map of topological spaces
$|\mathscr{X}| \to |X|$. We say $\pi$ is a \emph{stable} weak topological
moduli space if there is an open dense set of stable  points in $\mathcal{X}$.
\end{definition}
\begin{example}
If $\pi: \mathscr{X} \to X$ is an adequate moduli space, then the stable points of $\mathscr{X}$ are exactly those points that are not identified with other points of $\mathscr{X}$ when passing to the adequate moduli space $X$ \cite[Thm. 5.3.1 (5)]{ams}. In particular, the adequate moduli space $\mathscr{B}un_G^{ss} \to M_G^{ss}$ is stable when the genus $g_C \geq 2$, 
while the good moduli space 
$\Theta=[\mathbb{A}^1/\mathbb{G}_m] \to \mathrm{Spec}(k)$ is not.
\end{example}

\begin{lemma} 
\label{lemma: open substacks of adequate moduli spaces}
    Let $\mathscr{X}$ be a locally factorial irreducible quasi-separated stack
    of finite type over $k$ that admits a stable schematic weak topological
    moduli space $\pi: \mathscr{X} \to X$. Suppose that every effective Cartier
    divisor on $\mathscr{X}$ is, up to a power, the pullback of an effective
    Cartier divisor on $X$. 

    Let $\mathscr{U} \subseteq \mathscr{X}$ be an open substack that admits a
    separated weak topological moduli space $\eta: \mathscr{U} \to U$. Let $f: U \to X$
    denote the induced morphism. 
    Then the schematic locus of $U$ coincides with the $f$-saturated open where
    $f$ is an isomorphism onto its image.
\end{lemma}
\begin{proof}
    Both $U$ and $X$ are irreducible and normal (see \Cref{prop: properties of wtms}). 
    The morphism $f:U \to X$ is a separated
    birational morphism by the
    assumption that $\pi:\mathscr{X} \to X$ is stable weak topological moduli space. 
    Since $X$ is normal, by Zariski's main theorem for algebraic spaces  
    \cite[\href{https://stacks.math.columbia.edu/tag/082K}{Tag 082K}]{sp}
    the open locus where the fibers of $f$ are finite coincides with
    the locus where $f$ is an isomorphism onto its image.
    This locus is clearly contained in the schematic locus of $U$ since
    $X$ is a scheme.

    To see the converse inclusion let $x \in X(k)$ be a $k$-point such that
    $f^{-1}(x)$ is not finite. It remains to prove that every $k$-point in
    $f^{-1}(x)$ is not schematic in $U$. Choose a $k$-point $u \in
    f^{-1}(x)(k)$. By the definition of schematic fiber, we need to show that
    $\Fib_u(U) \neq \{u\}$. To see this, let $\widetilde{u} \in \mathscr{U}(k)$ be
    the unique closed $k$-point lying above $u$, so that
    $\Fib_{\tilde{u}}(\mathscr{U}) = \eta^{-1}(\Fib_u(U))$, 
    see \Cref{lemma: schematic fiber for stack with adequate moduli space}. 
    Our goal for the rest of the proof is to understand $\Fib_{\tilde{u}}(\mathscr{U})$. 
    We need the following:
    
    \noindent \textbf{Claim:} There is an equality of topological spaces $|\Fib_{\tilde{u}}(\mathscr{U})| = |\pi^{-1}(x) \cap \mathscr{U}|$.

    Let $\widetilde{x} \in \mathscr{X}(k)$ be
    the unique closed $k$-point lying above $x$ in $\mathscr{X}$. We shall show
    the following equality of subsets of the topological subspaces of
    $|\mathscr{U}|$:
    \begin{equation} \label{equation: equality topological spaces 1}
        \{ \, |\mathscr{Z} \cap \mathscr{U}| \, \mid \,  \mathscr{Z} \in \mathrm{Weil}_{\tilde{x}} \} = \{ \, |\mathscr{H}| \, \mid \, \mathscr{H} \in \mathrm{Weil}_{\tilde{u}}\},
    \end{equation}
    where the left-hand side $\mathrm{Weil}_{\tilde{x}}$ is taken in $\mathscr{X}$ whereas $\mathrm{Weil}_{\tilde{u}}$ is taken in $\mathscr{U}$.
By the base locus description of the schematic fiber (\Cref{prop:schematic fiber
as base locus}), we see that this implies the desired equality
$|\Fib_{\tilde{u}}(\mathscr{U})| = |\Fib_{\tilde{x}}(\mathscr{X}) \cap
\mathscr{U}| = |\pi^{-1}(x) \cap \mathscr{U}|$, where the last equality follows
the schematicity of $X$ by \Cref{prop: schematic points for adequate moduli
spaces}.

    We prove both inclusions of sets in \Cref{equation: equality topological spaces 1}. 
    Suppose that $\mathscr{Z} \in \mathrm{Weil}_{\tilde{x}}$. 
    Up to passing to a power of the divisor, which does not change the
    underlying topological space, we may assume that $\mathscr{Z}$ is the
    pullback of an effective divisor from $X$. 
    It follows that $\mathscr{Z}$ contains every point in $\mathscr{X}$ 
    that maps to $x$, which includes $\widetilde{u}$. 
    Furthermore, since $\tilde{u}$ is a generalization of $\tilde{x}$ (because it lies in the same fiber $\pi^{-1}(x)$), 
    any trivializing neighborhood of $\tilde{x}$
    for $\mathscr{Z}$ also contains $\tilde{u}$, 
    so $\mathscr{Z}$ is also trivializable around $\tilde{u}$. 
    We conclude that $\mathscr{Z} \cap \mathscr{U}$ 
    satisfies the conditions to be an element of $\mathrm{Weil}_{\tilde{u}}$. 

    Conversely, suppose that $\mathscr{H} \in \mathrm{Weil}_{\tilde{u}}$. Consider the closure $\overline{\mathscr{H}} \subseteq \mathscr{X}$. To conclude we show that, up to passing to a power of the divisor, we have $\overline{\mathscr{H}} \in \mathrm{Weil}_{\tilde{x}}$. Up to passing to a power, we know that $\overline{\mathscr{H}}$ is the pullback of a Cartier divisor on the scheme $X$, so it is Zariski-locally trivializable around any point of $\mathscr{X}$. Furthermore, since $\overline{\mathscr{H}}$ comes from $X$, it contains all the points that map to the same element $f(u)=x$ in $X$. In particular, $\overline{\mathscr{H}}$ contains $\tilde{x}$. Therefore, $\overline{\mathscr{H}}$ satisfies both of the necessary properties to be in $\mathrm{Weil}_{\tilde{x}}$. This concludes the proof of the \textbf{Claim}.

    Now we continue with the proof of the lemma. 
    On the one hand, we have
    \begin{equation} \label{equation: equality of schematic fiber with base locus 1}
        |\Fib_{\tilde{u}}(\mathscr{U})| = 
        \bigg{|}\mathrm{L}_{\tilde{u}}
        \bigg{(}\bigcap_{\mathscr{H} \in \mathrm{Weil}_{\tilde{u}}}
        \mathscr{H}\bigg{)}\bigg{|},
    \end{equation}
    by \Cref{prop:schematic fiber as base locus}.
    On the other hand, by the \textbf{Claim}, we have
    \[\bigg{|}\mathrm{L}_{\tilde{u}}\bigg{(}\bigcap_{\mathscr{H} \in
    \mathrm{Weil}_{\tilde{u}}} \mathscr{H}\bigg{)}\bigg{|} =
    \bigg{|}\mathrm{L}_{\tilde{u}}\bigg{(}\bigg{(}\bigcap_{\mathscr{Z} \in
    \mathrm{Weil}_{\tilde{x}}} \mathscr{Z}\bigg{)} \cap
    \mathscr{U}\bigg{)}\bigg{|} =
    \bigg{|}\mathrm{L}_{\tilde{u}}\bigg{(}\bigcap_{\mathscr{Z} \in
    \mathrm{Weil}_{\tilde{x}}} \mathscr{Z}\bigg{)} \cap \mathscr{U}\bigg{|}, \]
    where the last equality follows from the definition of $|\mathrm{L}_{\tilde{u}}(\bullet)|$ as the union of the irreducible components containing $\widetilde{u}$. Now, since $\widetilde{u}$ is a generalization of $\widetilde{x}$, 
    we see that every irreducible component of  $\bigcap_{\mathscr{Z} \in \mathrm{Weil}_{\tilde{x}}} \mathscr{Z}$ containing $\widetilde{u}$ also contains $\widetilde{x}$. Therefore, we may write
    \begin{equation} \label{equation: equality of two localizations}
        \mathrm{L}_{\tilde{u}}\bigg{(}\bigcap_{\mathscr{Z} \in
        \mathrm{Weil}_{\tilde{x}}} \mathscr{Z}\bigg{)} =
        \mathrm{L}_{\tilde{u}}\bigg{(}\mathrm{L}_{\tilde{x}}\bigg{(}\bigcap_{\mathscr{Z}
        \in \mathrm{Weil}_{\tilde{x}}} \mathscr{Z}\bigg{)}\bigg{)}.
    \end{equation}
    By \Cref{prop:schematic fiber as base locus}, we have $\mathrm{L}_{\tilde{x}}\left(\bigcap_{\mathscr{Z} \in \mathrm{Weil}_{\tilde{x}}} \mathscr{Z}\right) = \Fib_{\tilde{x}}(\mathscr{X})$, and since $x$ is sche\-ma\-tic, this is equal to $\pi^{-1}(x)$ (\Cref{prop: schematic points for adequate moduli spaces}). We conclude, by combining Equation 
    (\ref{equation: equality of schematic fiber with base locus 1}) and \Cref{equation: equality of two localizations}, that we have
    \[ |\Fib_{\tilde{u}}(\mathscr{U})| = \left| \mathrm{L}_{\tilde{u}}\left(\pi^{-1}(x) \cap \mathscr{U}\right) \right|.\]
    Now, note that $\pi^{-1}(x) \cap \mathscr{U}$ consists exactly of the points
    in $\mathscr{U}$ that map to $f^{-1}(x)$, in other words $\pi^{-1}(x) \cap
    \mathscr{U} = \eta^{-1}(f^{-1}(x))$. By Zariski's main theorem, every
    irreducible component of the preimage $f^{-1}(x)$ has positive dimension,
    and therefore $\mathrm{L}_{u}(f^{-1}(x))$ consists of positive dimensional
    irreducible components. Choose one component $Z \subseteq \mathrm{L}_{u}(f^{-1}(x))$. Since $\eta: \eta^{-1}(f^{-1}(x)) \to f^{-1}(x)$
    is surjective and universally closed, there is an irreducible component
    $\mathscr{J} \subseteq \eta^{-1}(f^{-1}(x))$ that surjects onto $Z$. 
    Then we obtain
    \[|\mathscr{J}| \subseteq |\mathrm{L}_{\tilde{u}}\left(\pi^{-1}(x) \cap \mathscr{U}\right)| = |\Fib_{\tilde{u}}(\mathscr{U})|.\]
    Since $\eta(\mathscr{J}) = Z \neq \{u\}$, we conclude $\eta(\Fib_{\tilde{u}}(\mathscr{U})) \neq \{u\}$. 
    This shows that we have $\Fib_u(U) \neq \{u\}$, 
    as $\Fib_{\tilde{u}}(\mathscr{U}) = \eta^{-1}(\Fib_u(U))$,
    i.e., $u$ is not a schematic point of $U$.
\end{proof}

\begin{theorem} \label{thm: open substacks of semistable G-bundles}
    Let $C$ be a smooth projective connected curve over $k$
    of genus at least $2$. Let $G$ be a connected reductive group over $k$. Fix
    $d \in \pi_1(G)$. Let $\mathscr{U} \subseteq \Bun_{G}^{d,ss}$ be an open
    substack of the stack of degree $d$ semistable $G$-bundles. Suppose that
    $\mathscr{U}$ admits a separated weak topological moduli space $\eta:
    \mathscr{U} \to U$. Let $f: U \to M^{d,ss}_G$ be the induced morphism to the
    moduli space of degree $d$ semistable $G$-bundles. 

    Then the schematic locus of $U$ coincides with the $f$-saturated
    open where $f$ is an isomorphism onto its image.
\end{theorem}

\begin{proof}
It suffices to show that the adequate moduli space $\pi: \Bun_G^{d,ss} \to
M_G^{d,ss}$ satisfies the hypotheses of \Cref{lemma: open substacks of adequate
moduli spaces}. First, since $g_C \geq 2$, the moduli
space $\pi$ is stable, because the stable locus is non-empty \cite[Prop. 3.25]{holla2002parabolicreductionsprincipalbundles}. 
Furthermore, the stack $\Bun_G$ is smooth.
It remains to show that effective Cartier divisors on
$\Bun_G^{d,ss}$ are, up to a power, pulled back from the moduli space $M^{d,ss}_G$. 
Since
$\pi_*(\cO_{\Bun_{G}^{d,ss}}) = \cO_{M^{d,ss}_G}$, by the projection formula it
suffices to show that if a line bundle $\cL$ on $\Bun_G^{d,ss}$ has a non-trivial
global section, then there is a power $\cL^{\otimes n}$ with $n>0$ that is
pulled back from $M^{d,ss}_G$.

To show this, note that the codimension of the complement $\Bun_G^{d} \setminus \Bun_G^{d,ss}$ of the semistable locus is at least $2$. Indeed, this is a 
known fact that can be seen via
the Harder-Narasimhan stratification of $\Bun_G^{d}$ and
calculating the codimension of the strata (see \cite[(10.7)]{Yang-Mills} or
\cite[Sect. 3]{Laumon-Rapoport-codim}). Hence the Picard group of
$\Bun_G^{d,ss}$ is the same as the Picard group of $\Bun_{G}^{d}$. The latter
has been described in \cite{biswas-hoffmann}. Let us now introduce some
necessary setup and notation in order to explain some of the relevant results
from \cite{biswas-hoffmann} that we require.

Let $\theta: G \twoheadrightarrow G^{ab}:= G/[G,G]$ 
denote the abelianization of $G$. 
We have a morphism of stacks 
$\theta: \Bun_{G}^{d} \to \Bun_{G^{ab}}^{\theta(d)}$, 
where $\theta(d)$ is the image of $d$ under the
morphism $\theta: \pi_1(G) \to \pi_1(G^{ab})$.
Recall that $G^{ab}$ is a torus isomorphic to $\mathbb{G}_m^{r_Z}$ for some
integer $r_Z \geq 0$, and the induced morphism from the center $Z_G \to G^{ab}$
is an isogeny. In particular, the stack $\Bun_{G^{ab}}^{\theta(d)}$ is a product
of Picard stacks, which has a good moduli space $M_{G^{ab}}^{\theta(d)}$
admitting a morphism $M^{d,ss}_{G} \to M_{G^{ab}}^{\theta(d)}$. 
A choice of base-point $p \in C(k)$ induces an identification 
$\Bun_{G^{ab}}^{\theta(d)} = M_{G^{ab}}^{\theta(d)} \times \rB\mathbb{G}^{r_Z}_m$. 

The adjoint group of $G$ is of the form $\prod_{i \in I} G_i$, 
where $G_i$ is a simple adjoint group. 
Denote by $q_i: G \to G_i$ the corresponding quotient morphism, 
and let $q_i(d)$ be the image of $d$ under the induced map $q_i: \pi_1(G) \to \pi_1(G_i)$. 
Then we have an induced morphism of stacks $\Bun_G^{d} \to \Bun_{G_i}^{d_i}$. 
Let us denote by $\cL_i$ the pullback to $\Bun_G^d$ of the determinant of cohomology line bundle 
$\cL_{\det,i}$ on $\Bun_{G_i}^{d_i}$ (as defined in \Cref{subsubsection: determinant of cohomology}). 
We claim that each $\cL_i$ descends to the adequate moduli space $M_G^{d,ss}$. 
To see this, first note that the morphism $\Bun_G^{d} \to \Bun_{G_i}^{d_i}$ 
is readily seen to be compatible with semistability. 
Therefore, it restricts to a morphism $\Bun_{G}^{d,ss} \to \Bun_{G_i}^{d_i,ss}$. 
By the universal property of adequate moduli spaces, 
we get a commutative diagram
\[
     \begin{tikzcd}
        \Bun_{G}^{d,ss} \ar[d] \ar[r] & \Bun_{G_i}^{d_i, ss} \ar[d]\\
         M_G^{d,ss} \ar[r] & M_{G_i}^{d_i,ss}.
     \end{tikzcd}
\]
Hence, in order to show that $\cL_i$ on $\Bun_{G}^{d,ss}$ descends to
$M_G^{d,ss}$ up to a power, it suffices to prove that $\cL_{\det,i}$ on
$\Bun_{G_i}^{d_i,ss}$ descends to $M_{G_i}^{d_i,ss}$ up to a power.

Since $M_{G_i}^{d_i,ss}$ is a projective variety of positive dimension 
(see \Cref{subsubsection: stability}),
there is a non-trivial ample line bundle $\cM_i$ on $M_{G_i}^{d_i,ss}$. 
The pullback $\pi_i^*(\cM_i)$ via the adequate moduli space morphism
$\pi_i: \Bun_{G_i}^{d_i,ss} \to M_{G_i}^{d_i,ss}$ is non-trivial by the projection formula. 
Combining some of the results in \cite{biswas-hoffmann}, we see that
$\pic(\Bun_{G_i}^{d_i,ss}) \otimes \mathbb{Q} = \pic(\Bun_{G_i}^{d_i}) \otimes
\mathbb{Q} = \mathbb{Q}$ 
(more precisely, using \cite[Thm. 5.3.1 (ii)]{biswas-hoffmann} we see that
$\pic(\Bun_{G_i}^{d_i}) \otimes \mathbb{Q} \cong 
\mathrm{NS}(\Bun_{G_i}^{d_i}) \otimes \mathbb{Q}$, 
and using \cite[Defn. 5.2.1 + Defn. 4.3.2 (i)]{biswas-hoffmann} we see that
$\mathrm{NS}(\Bun_{G_i}^{d_i}) \otimes \mathbb{Q} \cong \mathbb{Q}$).
This implies that every line bundle in 
$\pic(\Bun_{G_i}^{d_i,ss}) = \pic(\Bun_{G_i}^{d_i})$ is, 
up to a power, isomorphic to $\cL_{\det,i}$. 
We conclude that $\pi_i^*(\cM_i)^{\otimes n} \cong
\cL_{\det,i}^{\otimes m}$ for some non-zero integers $n,m$, and hence
$\cL_{\det,i}$ descends to the adequate moduli space up to a power, as desired.
This concludes the proof that $\cL_i$ descends, up to a power, to $M_G^{d,ss}$.

Now, it follows from \cite[Thm. 5.3.1 (ii)]{biswas-hoffmann}
that for every line bundle $\cL$ on $\Bun_G^d$, there exists some integer $N>0,$ such that we may write
\[ 
\cL^{\otimes N} = 
\theta^*(\cL_{ab} \otimes \cL_\rB) \otimes \bigotimes_{i \in I} \cL_i^{\otimes n_i},
\]
where $\cL_{ab} \in \pic(M_{G^{ab}}^d), \cL_\rB \in \pic(\rB\mathbb{G}_m^{r_Z}),$ and
$n_i\in\mathbb{Z}$.
To conclude the proof, 
we must show that if $\cL$ admits a non-trivial section, 
then $\cL^{\otimes N}$ descends to $M_{G}^{d,ss}$. 
In view of our discussion above, the line bundles 
$\theta^*(\cL_{ab})$ and $\cL_i^{\otimes n_i}$ 
descend to the adequate moduli space $M_G^{d,ss}$ up to a power. 
Hence, it suffices to show that the existence of a non-trivial section on $\cL^{\otimes N}$ 
implies that $\cL_\rB$ is trivial.

If $r_Z =0$, there is nothing to prove. 
If $r_Z>0$, then note that there is 
a non-trivial maximal central torus $Z_G^{\circ} \subseteq G$, 
and there is a canonical copy of $Z_G^{\circ}$ inside the automorphism group of
every geometric point $x \in |\Bun_G^d|$. 
On the one hand, it follows from the
definition of the line bundles that $Z_G^{\circ}$ acts trivially on the
$x$-fibers of the line bundles $\cL_{ab}$ and $\cL_i$ for any $i \in I$. 
On the other hand, if $\cL_\rB$ is non-trivial, 
then the isogeny $Z_G^{\circ} \to G^{ab}$ shows that $Z_{G^{ab}}^{\circ}$ 
acts non-trivially on $\cL_\rB$. 
In particular, if $\cL_\rB$ were non-trivial, 
then the subgroup of inertia $Z_G^{\circ}$ would act non-trivially on the fiber of
$\cL^{\otimes N}$ at every geometric point of $\Bun_G^d$, thus precluding the
existence of a non-trivial section of $\cL^{\otimes N}$, as desired.
\end{proof}

As the schematic locus of a quasi-separated algebraic space is a dense open subspace \cite[\href{https://stacks.math.columbia.edu/tag/086U}{Tag 086U}]{sp} we obtain
-- as an immediate corollary of \Cref{thm: semistable and schematic}
and \Cref{thm: open substacks of semistable G-bundles} --
that all moduli spaces of principal $G$-bundles are birational.
In particular, they have the same dimension.

\begin{corollary}
    \label{cor: moduli of principal bundles are birational}
    Let $C$ be a smooth projective connected curve over $k$ of genus at least $2$. 
    Let $G$ be a connected reductive group over $k$ and $d\in\pi_1(G)$.
    Furthermore, let $\mathscr{U}\subseteq \mathscr{B}un_G^d$ be a non-empty
    open substack
    admitting a weak topological moduli space $\mathscr{U} \to U$.
    Then $U$ is birational to the moduli space $M^{d,ss}_G$ 
    of semistable principal $G$-bundles of degree $d$ over $C$.
In particular, $\dim(U)=\dim(M^{d,ss}_G)$. %$=\dim(G)(g_C-1)+\dim (Z(G))$,    where $Z(G)$ denotes the center of $G$.
\end{corollary}
\begin{proof}
    Let $U_{sch}\subseteq U$ be the schematic locus of the algebraic space $U$, which is a dense open subspace. Replacing the scheme $U_{sch}$
    by an affine open subscheme we may assume that $U_{sch}$ is separated.
    Then the pre-image $\mathscr{U}_{sch}\subseteq \mathscr{U}$ of $U_{sch} \subseteq U$
    admits $U_{sch}$ as its weak topological moduli space.
    As $U_{sch}$ is a scheme we can first apply \Cref{thm: semistable and schematic} to obtain that $\mathscr{U}_{sch}\subseteq \mathscr{B}un^{d,ss}_G$.
    Then applying \Cref{thm: open substacks of semistable G-bundles} concludes
    that the induced morphism $U_{sch}\to M^{d,ss}_G$ is an open immersion.
\end{proof}

\bibliographystyle{alpha}

\bibliography{bibliography}

\end{document}